\documentclass[preprint,12pt]{article}

\usepackage{float}
\usepackage{amssymb}
\usepackage{mathrsfs}
\usepackage{enumerate}
\usepackage{subfigure}
\usepackage{graphicx}
\usepackage{colortbl,dcolumn}
\usepackage{tabularx}
\usepackage{amsmath}
\usepackage{psfrag}
\usepackage{booktabs}
\usepackage{cite}
\usepackage{amsthm}

\usepackage[top=1.2in, bottom=1.2in, left=1in, right=1in]{geometry}

\usepackage{subfigure}

\newcommand{\R}{\mathbb{R}}
\newcommand{\N}{\mathbb{N}}
\newcommand{\E}{\mathbb{E}}
\newcommand{\F}{\mathbb{F}}
\newcommand{\G}{\mathbb{G}}
\renewcommand{\H}{\mathbb H}
\renewcommand{\P}{\mathbb{P}}

\newcommand{\B}{\mathbb{B}}

\newcommand{\tr}{\operatorname{trace}}

\newcommand{\Q}{\mathbf Q}

\renewcommand{\AA}{\mathbb A}

\newcommand{\DD}{\mathcal D}
\newcommand{\GG}{\mathcal G}
\newcommand{\LL}{\mathcal L}
\newcommand{\OO}{\mathcal O}

\renewcommand{\P}{\mathbb P}

\newcommand{\MM}{\mathbb M}

\newcommand{\FFF}{\mathcal F}

\newcommand{\<}{\langle}
\renewcommand{\>}{\rangle}

\newcounter{RomanNumber}
\newcommand{\MyRoman}[1]{\setcounter{RomanNumber}{#1}
{\rm\Roman{RomanNumber}}}
\renewcommand{\tr}{{\rm Tr}}

\newtheorem{tm}{Theorem}[section]

\newtheorem{lm}{Lemma}[section]
\newtheorem{prop}{Proposition}[section]
\newtheorem{cor}{Corollary}[section]
\newtheorem{rk}{Remark}[section]

\newtheorem{ap}{Assumption}[section]

\newcommand{\sly}[1]{{\color{black} #1}}

\begin{document}
\title{Semi-implicit energy-preserving numerical schemes for stochastic wave equation via SAV approach}

       \author{
       Jianbo Cui\footnotemark[1],
        Jialin Hong\footnotemark[2],
         and
        Liying Sun\footnotemark[3],\\
       {\small
       \footnotemark[1]~Department of Applied Mathematics, The Hong Kong Polytechnic University, Hung Hom, Hong Kong.}\\
        {\small\footnotemark[2]~\footnotemark[3]~Institute of Computational Mathematics and Scientific/Engineering Computing,}
        \\{\small Academy of Mathematics and Systems Science, Chinese Academy of Sciences, }
        {\small Beijing 100190, P.R.China }\\
        {\small\footnotemark[3]~Institute of Applied Physics  and Computational Mathematics, Beijing, 100094, China.}
        }
       \maketitle
        \footnotetext{\footnotemark[3]Corresponding author: liyingsun@lsec.cc.ac.cn}

\maketitle

       \begin{abstract}
          {\rm\small In this paper, we propose and analyze semi-implicit numerical schemes for the stochastic wave equation (SWE) with general nonlinearity and multiplicative noise.   
          	These numerical schemes, called stochastic scalar auxiliary variable (SAV) schemes, are  constructed by transforming the considered SWE into a higher dimensional stochastic system with a stochastic SAV.
          	We prove that they can be solved explicitly and preserve the modified energy evolution law  and the regularity structure of the original system. 
          	These structure-preserving properties are the keys to overcoming the mutual effect of the noise and nonlinearity. 
          	By proving new regularity estimates of the introduced SAV,  
          	we establish the strong convergence rate of stochastic SAV schemes and the further fully-discrete schemes with the finite element method in spatial direction. 
          	To the best of our knowledge, this is the first result on the construction and strong convergence of semi-implicit energy-preserving {scheme}s for nonlinear SWE.}\\


\end{abstract}

\section{Introduction}
\label{int}
In recent years, the stochastic wave equation has been widely exploited to characterize 
the sound propagation in the sea,  
the dynamics of the primary current density vector field within the grey matter of the human brain, 
heat conduction around a ring, 
the dilatation of shock waves throughout the sun,
the vibration of a string under the action of stochastic forces, etc., (see, e.g., \cite{Dalang1962LNM,Etter2012Adv,Galka2008CN,Orsingher1982AIHPSB,Thomas2012JMP}). 
As an intrinsic property of such stochastic partial differential equation (SPDE), the energy evolution law describes the longtime behaviour of the original system and is a keystone to prove the well-posedness of SWE with complex nonlinearities (see, e.g., \cite{BRZEZNIAK20164157,C02,Chow06}). 
Due to the loss of {the explicit expression of the}  analytical solution, {a lot of} researchers have been concerned with numerical schemes to simulate SWE and {the associated} energy evolution law.

Up to now, various energy-preserving {scheme}s have been developed for solving  SWEs (see, e.g., \cite{ACLW16,BLM21,YC07,CLS13,CHJS19,HHS22}).  
For instance, the finite element method and the stochastic trigonometric scheme are shown in \cite{CLS13} to preserve the linear growth of the averaged energy for the linear SWE with additive noise.
With respect to the SWE with globally Lipschitz continuous coefficient and additive noise, \cite{BLM21} presents that the discontinuous Galerkin finite element method satisfies the trace formula. 
By means of the linear finite element approximation and a stochastic trigonometric method, \cite{ACLW16} designs a fully-discrete scheme for the  nonlinear SWE driven by multiplicative noise, and proves that its numerical solution satisfies an almost trace formula in  additive noise case.  \cite{HHS22} constructs several fully-discrete schemes which preserve the averaged energy evolution law for nonlinear SWE driven by multiplicative noise based on {the Pad\'e approximation}. 
For the SWE with superlinear coefficient and additive noise, \cite{CHJS19} proposes {an implicit} fully-discrete scheme by adopting the spectral Galerkin method and the averaged vector field method, and {establishes} {its strong convergence rate}, where the exponential integrability and energy-preserving property of the numerical solution play key roles.

Despite fruitful results on energy-preserving numerical schemes for {nonlinear} SWEs, one has to solve the large stochastic algebraic system by iterative methods  due to the implicitness of existing numerical schemes in general. In practice,  it  would  be  ideal  to  be  able  to  treat  the nonlinear term in SWE explicitly. However,  it is known that the classical explicit and semi-implicit numerical schemes for stochastic differential equations with superlinear coefficients suffer from the strong or weak divergence phenomenon (see, e.g., \cite{Cui21,HJK10}). To deal with this issue, the existing approach often exploits implicit schemes, or uses the tamed (or truncated) strategy to achieve explicit numerical schemes which could not preserve the \sly{energy evolution law}.
Up to now, there has not been any numerical scheme which is explicitly solvable and can preserve the averaged energy evolution law of nonlinear SWE. One main objective of the present work is to fill this gap for nonlinear SWE via a structure-preserving approach.


\sly{
	For deterministic conservative partial differential equations and gradient flow systems, the SAV approach has achieved a lot of successes in constructing semi-implicit numerical {\color{black} scheme}s which could preserve the energy decaying property or the modified energy conservation law (see, e.g., \cite{SX18,SXY2019SIAMR}). In contrast, less result on the SAV schemes is known for SPDEs.
	A natural question is how the noise influences the construction  and analysis of numerical schemes based on the SAV approach. It turns out that several new challenges appear in this aspect for nonlinear SWE. 
	First, the energy  suffers from the random effect and satisfies the stochastic energy evolution law, rather than the energy conservation law. 
	Second, the direct usage of SAV approach may not balance the diffusion coefficient and the new auxiliary variable, which brings troubles in constructing numerical schemes preserving the modified energy evolution law. 
	Last, due to the mutual effect of noise and nonlinearity on the stochastic auxiliary variable, the strong convergence analysis of stochastic SAV numerical schemes  is more complicated  than the standard error estimate. 
	
	In this paper, we first transform the considered SWE into a higher dimensional stochastic system (called the stochastic SAV reformulation) by introducing  a stochastic SAV.  By further exploiting the exponential Runge--Kutta methods and Runge--Kutta methods, as well as the spatial finite element method, we propose two novel kinds of semi-linear temporal semi-discrete stochastic SAV schemes and the further fully-discrete schemes. 
	Then we prove that the proposed stochastic SAV schemes could preserve the averaged energy evolution law for the stochastic SAV reformulation, and overcome the divergence issue in the superlinear coefficient case. 
	By establishing  new regularity estimates of the stochastic auxiliary variable, we prove that the temporal strong convergence order of  stochastic SAV numerical schemes is $1$ under the globally Lipschitz continuous  condition.  
	In particular, this result is even new for the deterministic SWE.  The convergence analysis of stochastic SAV numerical schemes in the non-globally Lipschitz case is beyond the scope of this paper, and will be studied in the future.

}

The rest of this paper is organized as follows.
Section \ref{sec2} presents an abstract formulation of SWE, and introduces its fundamental  properties. 
In Section \ref{sec3}, the semi-implicit energy-preserving numerical schemes based on the SAV approach for {SWEs} are proposed.
Section \ref{sec4} is devoted to  deducing  the strong error estimate for the proposed numerical schemes.
In Section \ref{sec5}, we present energy-preserving fully-discrete schemes based on the SAV approach and the finite element  method. 
Numerical experiments are carried out in Section \ref{sec6} to verify theoretical results.

\section{Stochastic wave equation}
\label{sec2}

In this section, we present the main assumptions and introduce the properties of the SWE with general nonlinearities and multiplicative noise, i.e.,
\begin{equation}
\left\{
\begin{aligned}
\label{mod;swe}
&du(t)=v(t)dt,\quad\quad \;\; &\text{in} \;\; \OO \times (0,T],\\
&dv(t)=\Lambda u(t)dt-f(u(t))dt+g(\Theta u(t))dW(t), \quad\quad \;\; &\text{in} \;\; \OO \times (0,T], \\
&u(0)=u_0, \quad v(0)=v_0, \quad \;\; &\text{in} \;\; \OO,
\end{aligned}
\right.
\end{equation}
where $\OO\in \mathbb R^{d},$ $d\le 3,$ is a bounded convex domain with polygonal boundary $\partial \OO$, $T\in(0,\infty),$  $u_0,v_0:\OO\rightarrow \R$ are $\mathcal F_0$-measurable, and $\Theta u(t):=(u(t), \partial_{x_1} u(t), \cdots, \partial_{x_d} u(t))$ with $\partial_{x_i}=\frac {\partial}{\partial x_i}, i\le d.$  
Here,  $\Lambda=\sum\limits_{i=1}^d\frac{\partial^2}{\partial x_i^2}$ is the Laplace operator with the homogeneous Dirichlet boundary condition with eigenpairs $\{(\lambda_i,e_i)\}_{i=1}^\infty,$ where $0< \lambda_1\leq\lambda_2\leq\cdots,$ and $\{e_i\}_{i=1}^\infty$ forms an orthonormal basis in $L^2:=L^2(\OO;\R).$  
Moreover, $W(\cdot)=\sum\limits_{k=1}^\infty\Q^{\frac12}e_k\beta_k(\cdot)$ is an $L^2$-valued $\Q$-Wiener process with respect to a filtered probability space $(\Omega, \FFF, \{\FFF_t\}_{t\geq 0}, \P)$, where $\{\beta_k\}_{k\in \N^+}$ is a sequence of \sly{i.i.d.} real-valued Brownian motions and $\Q$ is a symmetric, positive definite and finite trace operator.  
Denoting $L^p:=L^p(\mathcal O,\mathbb R),$ $p\ge 1,$ and $\dot\H^r:=\DD((-\Lambda)^\frac r2),$ $r\in\mathbb R,$ equipped with the inner product
$
\<x,y\>_{\dot{\H}^r}=\left\<(-\Lambda)^\frac r2x,(-\Lambda)^\frac r2y\right\>_{L^2}
=\sum\limits_{i=1}^\infty\lambda_i^r\<x,e_i\>_{L^2}$$\<y,e_i\>_{L^2},$ 
we introduce the mild assumptions on $f: \dot{\mathbb H}^1\to \dot{\mathbb H}$ and $g:  \dot{\mathbb H}^1\otimes\dot{\mathbb H}^{\otimes d} \to \mathcal L_2(\mathbf Q^{\frac 12} (\dot{\mathbb H}),\dot{\mathbb H})$, where $\dot{\mathbb H}:=\dot{\mathbb H}^0=L^2$ and $\mathcal L_2(\mathbf Q^{\frac 12} (\dot{\mathbb H}),\dot{\mathbb H})$ is the space of Hilbert--Schmidt operators from $\mathbf Q^{\frac 12} (\dot{\mathbb H})$ into $\dot{\mathbb H}$.

\begin{ap}
	\label{ap1}
	Assume that \sly{$\|\mathbf Q^{\frac 12}\|_{\mathcal L_2( \dot{\mathbb H},\dot{\mathbb H})}<\infty,$}
	\begin{align*}
		&\|f(u)\|_{\dot{\mathbb H}}+\|g(\Theta u)\|_{\mathcal L_2(\mathbf Q^{\frac 12} (\dot{\mathbb H}), \dot{\mathbb H})}\le c_0(\|u\|_{\dot{\H}^1}+1),\\
		&\|f(u)-f(\widetilde u)\|_{\dot{\mathbb H}}+\|g(\Theta u)-g(\Theta \widetilde u)\|_{\mathcal L_2(\mathbf Q^{\frac 12} (\dot{\mathbb H}),\dot{\mathbb H})}\le
		c_1\|u-\widetilde u\|_{\dot{\H}^1}
	\end{align*}
	for $u,\widetilde u \in \dot{\H}^1,$ where $c_0,c_1\ge 0$.
\end{ap}

\begin{ap}
	\label{ap2}
	Assume that 
	\begin{align*}
		&\|f(u)\|_{\dot{\mathbb H}}\le b_1(\|u\|_{\dot {\H}^1}),\\
		&\|f(u)-f(\widetilde u)\|_{\dot{\mathbb H}}\le b_2(\|u\|_{\dot {\H}^1},\|\widetilde u\|_{\dot {\H}^1})\|u-\widetilde u\|_{\dot {\H}^1},\\
		&\|g(\Theta u)\|_{\mathcal L_2(\mathbf Q^{\frac 12} (\dot{\mathbb H}),\dot{\mathbb H})}\le c_2(1+\|u\|_{L^{2k+2}}^{k+1}+\|u\|_{\dot{\H}^1}),\\
		&\|g(\Theta u)-g(\Theta \widetilde u)\|_{\mathcal L_2(\mathbf Q^{\frac 12} (\dot{\mathbb H}),\dot{\mathbb H})} \le b_3(\|u\|_{\dot{\H}^1},\|\widetilde u\|_{\dot{\H}^1})\|u-\widetilde u\|_{\dot{\H}^1},
	\end{align*}
	where $c_2\ge 0$, and $b_1,b_2,b_3$ are  polynomials with degree $2k+1,$ $2k$ and $k$, $k\in \mathbb N^+,$ respectively.  Furthermore, suppose that $\|\mathbf Q^{\frac 12}\|_{\mathcal L_2( \dot{\mathbb H},\dot{\mathbb H})}<\infty,$ and that there exist positive constants $c_3,c_4,\delta_0>0$ such that the potential $F(u)=\int_{\mathcal O} \widetilde F(u(x)) dx$ satisfies 
	\begin{align*}
		f(u)=\widetilde F'(u)\;\quad  \text{and}\; \quad F(u)+\delta_0 \ge c_3 \|u\|_{\dot{\mathbb H}}^2+c_4\|u\|_{L^{2k+2}}^{2k+2}.
	\end{align*}
\end{ap}

{
	\begin{rk}
		A typical example of $f$ satisfying Assumption \ref{ap2} is the Nemytskii operator of  $f(\xi)=\sum_{j=1}^{2k+1}a_j \xi^j $ with odd degree $2k+1$ and $a_{2k+1}>0$. More precisely, when $d\le 2,$ $k\in \mathbb N^+$ and when $d=3,$ $k=1$ (see, e.g., \cite{C02}).  It can be also found that Assumption \ref{ap1} is a special case of Assumption \ref{ap2}.
		
	\end{rk}
}

\subsection{Well-posedness and energy evolution law}
To present spatial and temporal  regularity estimates, as well as the averaged energy evolution law, of \eqref{mod;swe}, we let $X=(u,v)^\top$.
Then the compact form of  \eqref{mod;swe} reads
\begin{equation}
\label{mod;swe1}
dX(t)=AX(t)dt+\F(X(t))dt+\G(X(t))dW(t),\quad t\in(0,T],
\end{equation}
where 
\begin{align*}
	X(0)=X_0=
	\begin{bmatrix}
		u_0\\
		v_0
	\end{bmatrix},\quad
	A=
	\begin{bmatrix}
		0 & I\\
		\Lambda & 0
	\end{bmatrix},\quad
	\F(X(t))=
	\begin{bmatrix}
		0\\
		-f(u(t))
	\end{bmatrix},\quad
	\G=
	\begin{bmatrix}
		0\\
		g(\Theta u(t))
	\end{bmatrix}
\end{align*}
with $I$ being the identity operator defined in  $L^2$.  
Set the product space
$
\H^r:=\dot\H^r\times \dot\H^{r-1},$ $r\in \R,
$
endowed with the inner product
$
\<X_1,X_2\>_{\H^r}=\<x_1,x_2\>_{\dot\H^r}+\<y_1,y_2\>_{\dot\H^{r-1}}
$
for any $X_1=(x_1,y_1)^\top\in \H^r$ and $X_2=(x_2,y_2)^\top\in \H^r.$ 
It can been shown that the domain of operator $A$ is given by 
$$
\DD(A)=\left\{X\in\H:
AX=[v,\Lambda u]^\top
\in\H:=L^2\times \dot\H^{-1}\right\}=\dot\H^1\times L^2,$$ 
and $A$ generates a unitary group $E(t)=\exp (tA)=
\begin{aligned}
\begin{bmatrix}
\mathcal C(t)& (-\Lambda)^{-\frac 12}\mathcal S(t)\\
-(-\Lambda)^{\frac 12}\mathcal S(t) & \mathcal C(t)
\end{bmatrix}
\end{aligned},$ $t\in \R,$ 
where $\mathcal C(t)=\cos(t(-\Lambda)^{\frac 12})$ and $\mathcal S(t)=\sin(t(-\Lambda)^{\frac 12})$ are cosine and sine operators, respectively. 

Under  Assumption \ref{ap2}, one could follow the proof of \cite[Lemma 4.1 and Theorem 4.2]{C02} and get the global existence of the mild solution.

\begin{lm}[Existence and uniqueness of mild solution]
	\label{lm;eu}
	Let Assumption \ref{ap2}  hold and $X_0\in \mathbb H^1$. 
	Then there exists a unique mild solution of \eqref{mod;swe1}, i.e.,
	\begin{align}
		\label{sol;swe}
		X(t)=E(t)X_0+\int_0^tE(t-s)\mathbb F(X(s))ds+\int_0^t E(t-s)\mathbb G(X(s)) dW(s)
	\end{align}
	for $t \in[0, T].$ 
	Moreover, for any $p \in[2, \infty),$ there exists  $C(p,X_0,\Q,T)>0$ such that
	\begin{align*}
		\sup _{t \in[0, T]}\|X(t)\|_{L^{p}(\Omega; \mathbb H^1)} \leq C(p,X_0,\Q,T).
	\end{align*}
\end{lm} 

The Lyapunov energy functional $V_1:\mathbb{H}^1\rightarrow \mathbb{R}$ of \eqref{mod;swe1} is defined by
\begin{align}\label{Lya-V1}
	V_1(u,v)=\frac 12\|u\|^2_{\dot\H^1}+\frac 12\|v\|^2_{L^2}
	+\int_{\mathcal O} \widetilde F(u) dx,
\end{align}
and plays a crucial role in the study of \sly{well-posedness and blow-up} of \eqref{mod;swe1}. 
The following proposition introduces the averaged energy evolution law. 

\begin{prop}
	\label{lm;eel}
	Let Assumption \ref{ap2}  hold and $X_0\in \mathbb H^1$. 
	Then for $t\in [0,T],$ \eqref{mod;swe1} admits 
	\begin{align}\label{evo-law-ave}
		\E [V_1(u(t),v(t))]
		=&\E [V_1(u_0, v_0)]
		+\frac 12\int_0^t\E \Big[{\rm{Tr}}
		(g(\Theta u(s))\Q^\frac 12(g(\Theta u(s))\Q^\frac 12)^*)\Big]ds,
	\end{align}
	where ${\rm{Tr}}
	(g(\Theta u)\Q^\frac 12(g(\Theta u)\Q^\frac 12)^*)=\sum\limits_{i=1}^{\infty}\<g(\Theta u)\Q^\frac 12 e_i,g(\Theta u)\Q^\frac 12e_i\>_{L^2}$.
\end{prop}

\subsection{Regularity estimates}

In this part, we provide the properties of  $E(t),$ $t\in\mathbb R,$  and the regularity estimates of the mild solution of \eqref{mod;swe1}.    
The following lemma concerns with the temporal H\"{o}lder continuity of both sine and cosine operators, which {have} been discussed, for example, in \cite{ACLW16}.

\begin{lm}
	\label{lm;prpE_n}
	For $\gamma \in[0,1],$ there exists a positive constant {$C(\gamma)$} depending on $\gamma$, such that
	\begin{align*}
		&\|(\mathcal S(t)-\mathcal S(s))(-\Lambda)^{-\frac \gamma 2}\|_{\LL(\dot{\H})}\leq C(\gamma)(t-s)^\gamma,\\
		&\|(\mathcal C(t)-\mathcal C(s))(-\Lambda)^{-\frac \gamma 2}\|_{\LL(\dot{\H})}\leq C(\gamma)(t-s)^\gamma
	\end{align*}
	and
	$
	\|(E(t)-E(s))X\|_{\H}\leq \sly{C(\gamma)}(t-s)^\gamma \|X\|_{\H^\gamma}
	$ 
	for all $t\geq s\geq 0,$ {where $\LL(\dot{\H})$ denotes the space of all linear bounded operators from $\dot{\H}$ to $\dot{\H}$.}
\end{lm}

\begin{lm}
	\label{lm;2}
	For any $t\ge 0,$ 
	$\mathcal C(t)$ and $\mathcal S(t)$ satisfy a trigonometric identity in the sense that
	$
	\|\mathcal S(t)x\|_{L^2}^2+\|\mathcal C(t)x\|_{L^2}^2=\|x\|_{L^2}^2$ for $x\in L^2.$
\end{lm}

Based on the above trigonometric identity, it can be shown that $\|E(t)\|_{\LL(\H)}=1,$  $t\in\R.$  
By means of Lemmas \ref{lm;eu}-\ref{lm;2}, one could follow the similar arguments as in the proof of \cite[Lemma 3.3]{CHJS19} or  \cite[Theorem 7.4]{DZ92}, and obtain the following regularity estimates of the mild solution. 

\begin{prop}
	\label{holder}
	Let Assumption \ref{ap2}  hold and $T>0$. 
	Assume that $X_0\in \H^\beta$ for some $\beta \ge 1$ and that
	\begin{align*}
		&\|(-\Lambda)^{\frac {\beta-1}2} f(u)\|_{L^2}+\|(-\Lambda)^{\frac {\beta-1}2} g(\Theta u)\|_{\mathcal L_2(\mathbf Q^{\frac 12} (\dot{\mathbb H}), \dot{\mathbb H})}\le b_4(\|u\|_{\dot{\H}^1},\|u\|_{\dot{\H}^{\beta-1}}), \; u \in \dot{\H}^\beta,
	\end{align*}
	where  $b_4$ is a polynomial. Then it holds that for any $p\ge 2,$ 
	\begin{align*}
		\sup_{t\in [0,T]}\E \Big[\|X(t)\|_{\H^{\beta}}^p\Big]\le C(p,X_0,\Q,T),
	\end{align*}
	and for $0\le s\le t\le T,$
	\begin{align*}
		&\E \Big[\|u(t)-u(s)\|_{L^2}^p\Big]\le C(p,X_0,\Q,T)|t-s|^p, \\ 
		&\E \Big[\|v(t)-v(s)\|_{\dot{\mathbb H}^{-1}}^p\Big]\le C(p,X_0,\Q,T)|t-s|^{\frac {p}2}.
	\end{align*}
	{where  $C(p,X_0,\Q,T)$ is a positive constant depending on $p,X_0,\Q$ and $T$.} 
\end{prop}

Due to the loss of explicit expression of the exact solution, designing numerical schemes for nonlinear stochastic differential equations \sly{has} become an active area of research 
(see, e.g., \cite{CCDL20,CH17,CHLZ17b,KL18}). 
Below we shall focus on the construction and analysis of the temporal and fully-discrete numerical schemes for the considered SWE.

\section{SAV approach for stochastic wave equation}
\label{sec3}
The existing structure-preserving numerical schemes for the {nonlinear} SWE are implicit in general. 
When implementing such numerical schemes, one has to solve pathwisely complex stochastic algebraic systems, which often require a lot of computational costs.  
To avoid such issue, we introduce the equivalent form of \eqref{mod;swe} with an auxiliary variable and construct efficient  structure-preserving numerical schemes in this section.

\subsection{Stochastic wave equation with auxiliary variable}
\sly{In this part, we introduce} a scalar auxiliary variable 
$q:=\sqrt{F(u)+\delta_0}.$ 
Here, $\delta_0$ is a constant to make $q$ well-posed,  i.e., $F(u)+\delta_0\ge c> 0$ for any $u\in \dot{\mathbb H}^1$. 
\sly{Then} \eqref{mod;swe} can be reformulated into the following equivalent form
\begin{equation}
\label{mod;SAV}
\left\{\begin{aligned}
&d u(t)=v(t)dt,\quad  \;\; &\text{in} \;\; (0,T], \\
&d v(t)=\Lambda u(t)dt-\frac{f(u(t))}{\sqrt{F(u(t))+\delta_0}}q(t)dt+g(\Theta  u(t))dW(t),\quad  \;\; &\text{in} \;\; (0,T], \\
&dq(t)=\frac{\<f(u(t)), \partial_{t} u(t)\>_{L^2}}{2 \sqrt{F(u(t))+\delta_0}}dt, \;\; &\text{in} \;\; (0,T].
\end{aligned}\right.
\end{equation}

\begin{lm}
	Let Assumption \ref{ap2}  hold and $X_0\in \mathbb H^1$.  
	Then the averaged energy evolution law of $V(u(t),v(t),q(t)):=\frac 12\|u(t)\|^2_{\dot\H^1}+\frac 12\|v(t)\|^2_{L^2}
	+q(t)$ has the following form
	\begin{align}\label{aver-sav}
		&\E\Big[V(u(t),v(t),q(t))\Big]\nonumber\\
		=&\E \Big[V(u(0), v(0),q(0))\Big]
		+\frac 12\int_0^t \E \Big[{\rm{Tr}}
		(g(\Theta u(s))\Q^\frac 12(g(\Theta u(s))\Q^\frac 12)^*) \Big]ds,
	\end{align}
	\sly{where $t\in {[0,T]}.$} Furthermore, it holds that for $p\ge 2$ {and $0\leq s\leq t\leq T,$}
	\begin{align*}
		\E \Big[|q(t)-q(s)|^p\Big]&\le C(p,X_0,\Q,T)|t-s|^p,
	\end{align*}
	{ where  $C(p,X_0,\Q,T)$ is a positive constant depending on $p,X_0,\Q,$ and $T$.} 
\end{lm}

\begin{proof}
	Proposition \ref{lm;eel} implies the averaged  evolution law \eqref{aver-sav}. Notice that $$q(t)-q(s)=\int_s^t \frac {\<f(u(r)),v(r)\>_{L^2}}{\sqrt{F(u(r))+\delta_0}} dr.$$ 
	By using \sly{Assumption \ref{ap2},} H\"older's inequality and Lemma \ref{lm;eu}, we have  
	\begin{align*}
		\E \Big[|q(t)-q(s)|^p\Big]
		&\le \frac 1 {c^p} \int_{s}^t \E \Big[b_1^p(\|u(r)\|_{\dot{\H}^1})\|v(r)\|_{L^2}^p\Big]dr \le C(p,X_0,\Q,T)|t-s|^p,
	\end{align*}
	where $b_1$ is a polynomial.
\end{proof}

\subsection{Semi-implicit stochastic SAV {\color{black} scheme}s}
\label{sec;3}

Let $\tau$  be the time step-size, and denote $t_{n}=n\tau$ for $\sly{n\in\{1, \ldots,N\}}$ with $t_N=T.$ 
Let $u_n$ be the numerical approximation of $u(t_{n})$ for $n\in\{1, \ldots,N\}.$ We  
propose two novel kinds of {semi-implicit} energy-preserving numerical schemes via the SAV approach. 
In the rest of this paper, the constant $C$  may be different from line to line but never depending on $N$ and $\tau$.

The first natural choice is based on the Runge--Kutta type discretization and the SAV approach. 
\sly{To make it clear}, we employ the midpoint {\color{black} scheme} to approximate the linear unbounded part of stochastic system \eqref{mod;SAV} and obtain the midpoint SAV {\color{black} scheme}, i.e.,  
\begin{equation}
\label{nummet;SAV1}
\begin{aligned}
&u_{n+1}=u_n+\frac \tau 2(v_n+v_{n+1})+\frac \tau 2g(\Theta u_n)\delta W_n,\\
&v_{n+1}=v_n+\frac \tau 2\Lambda(u_n+u_{n+1})-\tau \frac{f(\widehat u_n)}{\sqrt{F(\widehat u_n)+\delta_0}}\frac{q_n+q_{n+1}}2+g(\Theta u_n)\delta W_n,\\
&q_{n+1}=q_n+\frac{\<f(\widehat u_n), u_{n+1}-u_n\>_{L^2}}{2 \sqrt{F(\widehat u_n)+\delta_0}}.
\end{aligned}
\end{equation}
Here,  $\widehat u_n=u_n$ \sly{(\text{or} $\frac{3u_n-u_{n-1}}2$  with $u_{-1}=u_0$)} , $n\in\{0,1, \ldots,N-1\},$ and $\delta W_n=W(t_{n+1})-W(t_n).$ 
We would like to remark that the modified term $\frac \tau 2g(\Theta u_n)\delta W_n$ is used to balance the diffusion coefficient and the stochastic auxiliary variable. Without this modification, the direct discretization of \eqref{mod;SAV} 
fails to  preserve the modified energy evolution law. 

An alternative way to construct semi-linear stochastic SAV {\color{black} scheme}s is inspired by the phenomenon that the exponential Runge--Kutta {\color{black} scheme}s often require less restriction for SPDEs (see, e.g., \cite{ACLW16,BCH18,CCHS20}) with regard to optimal strong convergence rate. By employing the exponential Euler {\color{black} scheme} to linear unbounded part of stochastic system \eqref{mod;SAV}, we get
\begin{equation}
\label{nummet;SAV2}
\begin{aligned}
&X_{n+1}=E(\tau)X_{n}+A^{-1}(E(\tau)-I)\Big(0,-\frac{f(\widehat u_n)}{\sqrt{F(\widehat u_n)+\delta_0}}\frac{q_n+q_{n+1}}2\Big)^\top\\
&\quad\quad\quad+
E(\tau)(0,g(\Theta u_n)\delta W_n)^\top,\\
&q_{n+1}=q_n+\frac{\<f(\widehat u_n), u_{n+1}-u_n\>_{L^2}}{2 \sqrt{F(\widehat u_n)+\delta_0}},
\end{aligned}
\end{equation}
where $X_{n+1}=(u_{n+1},v_{n+1})^\top,$ $n\in\{0,1,\ldots, N-1\}.$
We would like to mention that the proposed {\color{black} scheme}s \eqref{nummet;SAV1} and \eqref{nummet;SAV2} have a great potential  for approximating  other SPDEs, such as stochastic nonlinear Schr\"odinger equations and parabolic SPDEs. 

One main advantage of the proposed  {\color{black} scheme}s is that they can be solved efficiently. 
We take \eqref{nummet;SAV2} as an example to illustrate this fact. Notice that  \eqref{nummet;SAV2} can be rewritten as
\begin{align}
	u_{n+1}=&e_{11} u_{n}+e_{12} v_{n}-a^{1} \frac{f(\widehat u_n)}{\sqrt{F(\widehat u_n)+\delta_0}}\frac{q_n+q_{n+1}}2+e_{12}g(\Theta u_n)\delta W_n, \label{e1}\\
	v_{n+1}=&e_{21} u_{n}+e_{22} v_{n}-a^{2} \frac{f(\widehat u_n)}{\sqrt{F(\widehat u_n)+\delta_0}}\frac{q_n+q_{n+1}}2+e_{22}g(\Theta u_n)\delta W_n, \label{e2}\\
	q_{n+1}=&q_{n}+\frac{\<f(\widehat u_n), u_{n+1}-u_n\>_{L^2}}{2 \sqrt{F(\widehat u_n)+\delta_0}}\label{e3}
\end{align}
with $n\in\{0,1,\ldots, N-1\}$, $e_{11}=e_{22}=\mathcal C(\tau),$ $e_{12}=(-\Lambda)^{-\frac 12}\mathcal S(\tau),$ $e_{21}=-(-\Lambda)^{\frac 12}\mathcal S(\tau),$ $a^{1}=-(-\Lambda)^{-1}$ $(\mathcal C(\tau)-I)$  and $a^{2}=(-\Lambda)^{-\frac 12}\mathcal S(\tau).$ 
By eliminating $\frac{q_n+q_{n+1}}2$, we deduce 
\begin{align}
	\label{ite1}
	u_{n+1}+\gamma_n\< f(\widehat u_n), u_{n+1}\>_{L^2}=\Gamma_n,
\end{align}
where $\gamma_n=\frac{a^{1} f(\widehat u_n)}{4(F(\widehat u_n)+\delta_0)}$ and
\begin{align*}
	&\Gamma_n=e_{11} u_{n}+e_{12} v_{n}-a^{1} \frac{f(\widehat u_n)}{\sqrt{F(\widehat u_n)+\delta_0}}q_{n}+\gamma_n\< f(\widehat u_n), u_{n}\>_{L^2}+e_{12}g(\Theta u_n)\delta W_n.
\end{align*}
{Taking the inner product of \eqref{ite1} with $f(\widehat u_n),$ we} derive
\begin{align*}
	\left(1+\<f(\widehat u_n), \gamma_n\>_{L^2}\right)\< f(\widehat u_n), u_{n+1}\>_{L^2}=\< f(\widehat u_n), \Gamma_n\>_{L^2}.
\end{align*}
Here $\<f(\widehat u_n), \gamma_n\>_{L^2}\geq 0$ due to  $a^{1}=(-\Lambda)^{-1}(I-\mathcal C(\tau))$. Then it follows that
\begin{align*}
	\< f(\widehat u_n), u_{n+1}\>_{L^2}=\frac{\< f(\widehat u_n), \Gamma_n\>_{L^2}}{1+\<f(\widehat u_n), \gamma_n\>_{L^2}}.
\end{align*}
After solving $\< f(\widehat u_n), u_{n+1}\>_{L^2}$ from the linear system, $u_{n+1}$ is then updated from \eqref{ite1}. 
Subsequently, $q_{n+1}$ is obtained from \eqref{e3}. 
Finally, we get $v_{n+1}$ from \eqref{e2}.
Utilizing the similar procedures, one can {also} verify that $u_{n+1},$ $q_{n+1}$ and $v_{n+1}$ in \eqref{nummet;SAV1} can be solved  explicitly. 

\sly{Next we will show that the proposed schemes could overcome the divergent issue in the superlinear coefficient case and preserve the energy evolution law, which, together with the regularity estimates of the discrete auxiliary variable, plays the key role in the strong convergence analysis \ref{sec4}.}

\subsection{\sly{The preservation of energy evolution law}}
In this part, we address the important property of the proposed {\color{black} scheme}s, i.e., they could inherit the averaged energy evolution law of the SWE \eqref{mod;swe} (or \eqref{mod;swe1}). 
The rigorous proof of these results requires some regularization technique\sly{s}, like taking advantage of the spectral Galerkin approximation, such that the integration by parts formula {makes} sense, and then taking limit. For simplicity, we omit these tedious procedures.

\begin{prop}\label{ene-pre-1}
	Let Assumption \ref{ap2}  hold and $X_0\in \mathbb H^1$. 
	Numerical {\color{black} scheme}s \eqref{nummet;SAV1} and \eqref{nummet;SAV2} preserve the discrete averaged modified energy evolution law, for $n\in\{0,1,\ldots,N-1\},$
	\begin{align*}
		\mathbb E[V(u_{n+1},v_{n+1},q_{n+1})]=&\mathbb E[V(u_n,v_n,q_n)]+\frac \tau 2 \E \Big[{\rm{Tr}}
		\left(g(\Theta u_n)\Q^\frac 12(g( \Theta u_n)\Q^\frac 12)^*\right)\Big],
	\end{align*}
	where $V(u_n,v_n,q_n)=\frac 12\|u_n\|^2_{\dot\H^1}+\frac 12\|v_n\|^2_{L^2}+q_n^{2}.$
\end{prop}

\begin{proof}
	To simplify the presentation, we only present the proof of  \eqref{nummet;SAV2} since that of  \eqref{nummet;SAV1} is similar. 
	{Fix $n\in\{0,1,\ldots,N-1\},$
		and denote}
	$M=\begin{bmatrix}
	-\Lambda & 0\\
	0 & I
	\end{bmatrix},$  
	$Id=\begin{bmatrix}
	I & 0\\
	0 & I
	\end{bmatrix},$ 
	$J=\begin{bmatrix}
	0 & I\\
	-I & 0
	\end{bmatrix},$ $\tilde g_n=\begin{bmatrix}
	0\\g(\Theta u_n)\delta W_n
	\end{bmatrix}$ and  
	$\tilde f_n=
	\begin{bmatrix}
	0\\
	-\frac{f(\widehat u_n)}{\sqrt{F(\widehat u_n)+\delta_0}}\frac{q_n+q_{n+1}}2
	\end{bmatrix}.$ 
	Then using the integration by parts formula, it follows that
	\begin{align}\label{integ}
		\frac 12 \|u_n\|_{\dot{\H}^1}^2+\frac 12 \|v_n\|_{L^2}^2
		=\frac 12 \int_\OO X_{n}^\top MX_{n}dx.
	\end{align} 
	Notice that 
	\begin{align*}
		&X_{n+1}^\top MX_{n+1}-X_n^\top MX_n\\
		=&X_n^\top(E(\tau)^\top ME(\tau)-M)X_n+2X_n^\top E(\tau)^\top M(E(\tau)-Id)A^{-1}\tilde f_n\\
		&+2X_n^\top E(\tau)^\top ME(\tau)\tilde g_n
		+(\tilde f_n)^\top (A^{-1})^\top(E(\tau)-Id)^\top M(E(\tau)-Id)A^{-1}\tilde f_n\\
		&+2(\tilde f_n)^\top (A^{-1})^\top(E(\tau)-Id)^\top ME(\tau)\tilde g_n+ (\tilde g_n)^\top E(\tau)^\top ME(\tau)\tilde g_n, 
	\end{align*}
	and 
	\begin{align*}
		&q_{n+1}^2-q_n^2=\frac{\<f(\widehat u_n), u_{n+1}\>_{L^2}}{\sqrt{F(\widehat u_n)+\delta_0}}\frac{q_n+q_{n+1}}2
		-\frac{\<f(\widehat u_n), u_{n}\>_{L^2}}{\sqrt{F(\widehat u_n)+\delta_0}}\frac{q_n+q_{n+1}}2\\
		=&\int_\OO (X_n)^\top (E(\tau)-Id)^\top J^{-1}\tilde f_n dx
		+\int_\OO (\tilde f_n)^\top (A^{-1})^\top(E(\tau)-Id)^\top J^{-1}\tilde f_n dx\\
		&+\int_\OO (\tilde g_n)^\top E(\tau)^\top  J^{-1}\tilde f_n dx.
	\end{align*}
	As a consequence, due to $MA^{-1}=J^{-1}$ and the fact that $\int_{\mathcal O} \tilde f_n^{\top} (A^{-1})^{\top} M E(\tau) \tilde g_n dx= \int_{\mathcal O} \tilde g_n^{\top} E(\tau)^{\top}M A^{-1} \tilde f_n dx$,  we obtain 
	\begin{align*}
		&\frac 12 \int_\OO (X_{n+1}^\top MX_{n+1}-X_n^\top MX_n)dx+q_{n+1}^2-q_n^2\\
		=&\frac 12 \int_\OO X_n^\top(E(\tau)^\top ME(\tau)-M)X_n dx+\frac 12\int_\OO(\tilde g_n)^\top E(\tau)^\top ME(\tau)\tilde g_ndx\\
		&+\int_\OO X_n^\top\left(E(\tau)^\top M(E(\tau)-Id)A^{-1}+(E(\tau)-Id)^\top J^{-1}\right)\tilde f_ndx\\
		&+\int_\OO X_n^\top E(\tau)^\top ME(\tau)\tilde g_ndx\\
		&+\frac 12\int_\OO(\tilde f_n)^\top (A^{-1})^\top(E(\tau)-Id)^\top M \left(E(\tau)+Id\right)A^{-1}\tilde f_n dx\\
		&+\int_\OO (\tilde f_n)^\top (A^{-1})^\top\left((E(\tau)-Id)^\top ME(\tau)+ME(\tau)\right)\tilde g_ndx.
	\end{align*}
	From  $(\tilde f_n)^\top(A^{-1})^\top M \tilde g_n=0$, it follows that
	\begin{align}\nonumber 
		&\frac 12 \int_\OO (X_{n+1}^\top MX_{n+1}-X_n^\top MX_n)dx+q_{n+1}^2-q_n^2\\\nonumber 
		=&\frac 12 \int_\OO (X_n+A^{-1}\tilde f+\tilde g_n)^\top (E(\tau)^\top ME(\tau)-M)(X_n+A^{-1}\tilde f+\tilde g_n)dx\\\label{evo-ene-as1}
		&+\frac 12\int_\OO(\tilde f_n)^\top (A^{-1})^\top(E(\tau)^\top M-ME(\tau))A^{-1}\tilde f_n dx\\\nonumber 
		&+\int_\OO X_n^\top M \tilde g_ndx+\frac 12\int_\OO(\tilde g_n)^\top  M\tilde g_ndx.
	\end{align}
	According to the unitary property of $E(\tau)$ and the 
	skew symmetry of the matrix $E(\tau)^\top M-ME(\tau),$ taking expectation leads to
	\begin{align*}
		&\mathbb E(\frac 12\|u_{n+1}\|^2_{\dot\H^1}+\frac 12\|v_{n+1}\|^2_{L^2}+q_{n+1}^{2})\\
		=&
		\mathbb E(\frac 12\|u_n\|^2_{\dot\H^1}+\frac 12\|v_n\|^2_{L^2}+q_n^{2})+\frac \tau 2{\rm{Tr}}
		\left(g(\Theta u_n)\Q^\frac 12(g(\Theta u_n)\Q^\frac 12)^*\right),
	\end{align*}
	which completes the proof.
\end{proof}

As a result of Proposition  \ref{ene-pre-1}, we are in a position to present a priori estimates of numerical solutions.

\begin{cor}\label{pri-num} Let Assumption \ref{ap2}  hold and $X_0\in \mathbb H^1$.  In addition suppose that 
	$\|g(\Theta u)\|_{\mathcal L_2({\bf {Q}}^{\frac 12} (\dot{\H}), \dot{\H})}\le c_2(1+\|u\|_{\dot{\H}^1})$ with $c_2>0.$
	Then it holds that for any $p\ge1$,
	\begin{align*}
		\sup_{n\le N}\E \big[ V^p(u_n,v_n,q_n)\big]\le  C(p,X_0,\Q,T),
	\end{align*}
	where $(u_n,v_n,q_n)$ is the numerical solution of  \eqref{nummet;SAV1} or \eqref{nummet;SAV2}, and $C(p,X_0,\Q,T)>0$.
\end{cor}

\begin{proof}
	By \eqref{evo-ene-as1}, $\|g(\Theta u)\|_{\mathcal L_2( {\bf {Q}}^{\frac 12} (\dot{\H}), \dot{\H})}\le c_2(1+\|u\|_{\dot{\H}^1})$,  the Burkh\"older inequality, as well as H\"older's and Young's inequalities, we get that for $p\ge 1,$
	\begin{align*}
		&\quad\E [V^p(u_{n+1},v_{n+1},q_{n+1})]\\
		&\le \E [V^p(u_{n},v_{n},q_{n})](1+C(p)\tau)+p \E[V^{p-1}(u_{n},v_{n},q_{n})\<v_n,g(\Theta u_n)\delta W_n \>_{L^2}]\\
		&+C(p)\E\Big[ \<g(\Theta u_n)\delta W_n, g(\Theta u_n)\delta W_n\>_{L^2}^p\Big]
		+C(p)\tau \E \Big[ \|v_n\|_{L^2}^{2p}\Big]+C(p)\tau \E \Big[1+\|u_n\|_{\dot{\H}^1}^{2p}\Big]\\
		&\le \E [V^p(u_{n},v_{n},q_{n})](1+C(p)\tau)
		+C(p)\tau \E \Big[ \|v_n\|_{L^2}^{2p}\Big]+C(p)\tau \E \Big[1+\|u_n\|_{\dot{\H}^1}^{2p}\Big],
	\end{align*}
	where we have used the independent increment property of  $W(\cdot)$ and Proposition \ref{ene-pre-1} in the last inequality. 
	The discrete Gr\"onwall's inequality yields the desired result. 
\end{proof}

\section{Strong convergence rate of stochastic SAV {{\color{black} scheme}s}}
\label{sec4}
In this section, we provide a generic approach to study the strong convergence rates of the proposed stochastic SAV {\color{black} scheme}s \eqref{nummet;SAV1}-\eqref{nummet;SAV2}. 

\subsection{Properties of the discrete auxiliary variable}
In the following, we first show a priori estimates of {both $q_n$ and $u_n$ for ${n\in \{1,\ldots,N\}}$,} which are of vital importance in the study of strong error estimates {of the proposed stochastic energy-preserving {\color{black} scheme}s} via the SAV approach.

\begin{lm}
	\label{nu;holder}
	Under the condition of Corollary \ref{pri-num}, it holds that
	\begin{align*}
		{\sup\limits_{j\in\{0,1,\ldots, N-1\}}}\mathbb E[\|u_{j+1}- u_{j}\|_{L^2}^p]\leq C\tau^p,
	\end{align*}
	where $u_j$ is the numerical solution of \eqref{nummet;SAV1} or \eqref{nummet;SAV2}, $N\in \mathbb N^+, N\tau=T,$ and $C:=C(p,X_0,\Q,T)>0$.
\end{lm}
\begin{proof}
	{Fix $j\in\{0,1,\ldots, N-1\}.$}
	Based on the properties that $|\cos(x)-1|\leq C|x|^2$ and $|\frac {\sin(x)}{x}|\le C$ for some $C>0,$ we have
	\begin{align*}
		&\|u_{j+1}- u_{j}\|_{L^2}\\
		\leq & 
		\|(\cos (\tau(-\Lambda)^{\frac 12}-I)u_{j}\|_{L^2}+\|(-\Lambda)^{-\frac 12}\sin(\tau (-\Lambda)^{\frac 12})v_{j}\|_{L^2}\nonumber\\
		&+\|(-\Lambda)^{-1}(\cos (\tau (-\Lambda)^{\frac 12}-I)\frac{f(\widehat{u}_{j})}{\sqrt{F(\widehat{u}_{j})+\delta_0}}\frac{q_j+q_{j+1}}2\|_{L^2}\nonumber\\
		&+\|(-\Lambda)^{-\frac 12}\sin(\tau(-\Lambda)^{\frac 12})g(\Theta u_{j})\delta W_{j}\|_{L^2}\\\nonumber
		\leq &C\tau (\|u_{j}\|_{\dot{\mathbb H}^1}+\|v_{j}\|_{L^2})+C\tau^2 \frac{\|f(\widehat{u}_{j})\|_{L^2}}{\sqrt{F(\widehat{u}_{j})+\delta_0}} (|q_{j}|+|q_{j+1}|)+\tau \|g(\Theta u_{j})\delta W_{j}\|_{L^2}.\nonumber
	\end{align*}
	Thanks to $\widehat u_j=u_j$ or $\frac {3u_j-u_{j-1}}2$, taking the $p$th moment, and using Corollary \ref{pri-num}, Young's and H\"older's inequalities, we \sly{obtain} 
	\begin{align*}
		&\quad\E \Big[\|u_{j+1}- u_{j}\|_{L^2}^p \Big]\\
		&\le C(p)\tau^p \Big(\E [\|u_{j}\|_{\dot{\mathbb H}^1}^p]+\E [\|v_{j}\|_{\dot{\mathbb H}}^p]+\E [b(\|u_j\|_{\dot{\H}^1},\|\widehat u_{j}\|_{\dot{\H}^1})]+\E[q_j^{2p}+q_{j+1}^{2p}]\Big)\\
		&\le C(p,X_0,\Q,T)\tau^p, 
	\end{align*}
	where $b$ is a polynomial. 
	This completes the proof.
\end{proof}

\begin{prop}\label{aux-prop} 
	{Let the condition of Corollary \ref{pri-num} hold and} $ |q_0-\sqrt{F(\widehat u_0)+\delta_0}|\le C\tau$. Suppose that $$|\<f'(u)v,w\>|\le b_5(\|u\|_{\dot{\H}^1},\|v\|_{\dot{\H}^1},\|w\|_{\dot{\H}^1})\|v\|_{L^2}^{\gamma_1}\|w\|_{L^2}^{\gamma_1}$$ for some $\gamma_1>\frac 12$ and some polynomial $b_5.$
	Then for  \eqref{nummet;SAV1} and \eqref{nummet;SAV2}, it holds that
	\begin{align*}
		{\sup\limits_{j\in\{1,\ldots, N\}}}\mathbb E\Big[|\sqrt{F(\widehat u_{j})+\delta_0}-q_{j}|^p\Big]\leq C\tau^{\min(1,2\gamma_1-1)p},
	\end{align*}
	where $C:=C(p,X_0,\Q,T)>0$, $N\in \mathbb N^+, N\tau=T.$
\end{prop}
\begin{proof}
	{Fix $j\in\{1,\ldots, N-1\}.$}
	The definitions of $q_j$ and $F(\widehat u_j)$, and the Taylor expansion yield that 
	\begin{align*}
		&\sqrt{F(\widehat u_{j+1})+\delta_0}-q_{j+1}\\=&\sqrt{F(\widehat u_j)+\delta_0}-q_j-\tau \frac{\<f(\widehat u_j), \frac{u_{j+1}-u_j}\tau \>_{L^2}}{2 \sqrt{F(\widehat u_j)+\delta_0}}
		+\frac{F(\widehat u_{j+1})-F(\widehat u_{j})}{\sqrt{F(\widehat u_{j+1})+\delta_0}+\sqrt{F(\widehat u_j)+\delta_0}}\\
		=&\sqrt{F(\widehat u_j)+\delta_0}-q_j-\tau \frac{\<f(\widehat u_j), \frac{u_{j+1}-u_j}\tau \>_{L^2}}{2 \sqrt{F(\widehat u_j)+\delta_0}}+\tau \frac{\<f(\widehat u_j), \frac{\widehat u_{j+1}-\widehat u_j}\tau \>_{L^2}}{\sqrt{F(\widehat u_j)+\delta_0}+\sqrt{F(\widehat u_{j+1})+\delta_0}}\\
		&+\tau \frac{\<\int_0^1f'(\widehat u_j+\theta(\widehat u_{j+1}-\widehat u_j)d\theta (\widehat u_{j+1}-\widehat u_j), \frac{\widehat u_{j+1}-\widehat u_j}\tau \>_{L^2}}{\sqrt{F(\widehat u_j)+\delta_0}+\sqrt{F(\widehat u_{j+1})+\delta_0}}.
	\end{align*}
	By applying H\"older's inequality, we obtain 
	\begin{align*}
		&|\sqrt{F(\widehat u_{j+1})+\delta_0}-q_{j+1}|\\
		\leq &
		|\sqrt{F(\widehat u_{j})+\delta_0}-q_{j}|\\
		&+\tau |\<f(\widehat u_j), \frac{\widehat u_{j+1}-\widehat u_j}\tau \>_{L^2}|\Big|
		\frac{\sqrt{F(\widehat u_{j+1})+\delta_0}-\sqrt{F(\widehat u_{j})+\delta_0}}{2(F(\widehat u_{j})+\delta_0)+2\sqrt{F(\widehat u_{j})+\delta_0}\sqrt{F(\widehat u_{j+1})+\delta_0}}
		\Big|\\
		&+\tau \frac{\Big|
			\<\int_0^1f'(\widehat u_j+\theta(\widehat u_{j+1}-\widehat u_j)d\theta (\widehat u_{j+1}-\widehat u_j), \frac{\widehat u_{j+1}-\widehat u_j}\tau \>_{L^2}
			\Big|}{\sqrt{F(\widehat u_j)+\delta_0}+\sqrt{F(\widehat u_{j+1})+\delta_0}}\\
		=:&|\sqrt{F(\widehat u_{j})+\delta_0}-q_{j}|+A_1+A_2.
	\end{align*}
	Due to the assumption of $f,$ there exists a polynomial $b$ such that 
	\begin{align*}
		A_1\leq& \tau |\<f(\widehat u_j),\frac{\widehat u_{j+1}-\widehat u_j}\tau \>_{L^2}|
		\Big|
		\frac{F(\widehat u_{j+1})-F(\widehat u_{j})}
		{
			2\sqrt{F(\widehat u_{j})+\delta_0}
			(\sqrt{F(\widehat u_{j+1})+\delta_0}
			+\sqrt{F(\widehat u_{j})+\delta_0})^2
		}
		\Big|\\
		\leq & C b(\|\widehat u_j\|_{\dot{\H}^1},\|\widehat u_{j+1}\|_{\dot{\H}^1})\|\widehat u_{j+1}-\widehat u_{j}\|_{L^2}^2.
	\end{align*}
	For the term $A_2,$ using $|\<f'(u)v,w\>|\le b_5(\|u\|_{\dot{\H}^1},\|v\|_{\dot{\H}^1},\|w\|_{\dot{\H}^1})\|v\|_{L^2}^{\gamma_1}\|w\|_{L^2}^{\gamma_1}$, it holds that
	\begin{align*}
		A_2\leq &C b_5(\|\widehat u_j\|_{\dot{\H}^1},\|\widehat u_{j+1}\|_{\dot{\H}^1})  \|\widehat u_{j+1}-\widehat u_j\|_{L^2}^{2\gamma_1}.
	\end{align*}
	In sum, 	
	\begin{align*}
		&|\sqrt{F(\widehat u_{j+1})+\delta_0}-q_{j+1}|\\
		\leq &
		|\sqrt{F(\widehat u_{j})+\delta_0}-q_{j}|
		+C\Big(b(\|\widehat u_j\|_{\dot{\H}^1},\|\widehat u_{j+1}\|_{\dot{\H}^1})+b_5(\|\widehat u_j\|_{\dot{\H}^1},\|\widehat u_{j+1}\|_{\dot{\H}^1})  \Big)\|\widehat u_{j+1}-\widehat u_{j}\|_{L^2}^{2\gamma_1}\\
		&\times (1+\|\widehat u_{j+1}-\widehat u_{j}\|_{L^2}^{2-2\gamma_1}).
	\end{align*}
	Taking the $p$th moment for $p\ge 1$, using Lemma  \ref{nu;holder}, Corollary \ref{pri-num} and the definition of $\widehat u_j$, we conclude that 
	\begin{align*}
		\|\sqrt{F(\widehat u_{j+1})+\delta_0}-q_{j+1}\|_{L^p(\Omega)}&\le \|q_0-\sqrt{F(\widehat u_0)+\delta_0}\|_{L^p(\Omega)}+C \sum_{k=0}^j \|\widehat u_{j+1}-\widehat u_{j}\|_{L^{4p\gamma_1}(\Omega)}^{2\gamma_1}\\
		&\le C\tau +C\tau^{2\gamma_1-1}\le C \tau^{2\gamma_1-1}.
	\end{align*}

\end{proof}

It can be seen that if Assumption \ref{ap1} holds, one could take $\gamma_1=1$ in Proposition \ref{aux-prop}. 

\subsection{Strong convergence analysis of stochastic SAV {{\color{black} scheme}s}}
Now, we are in a position to prove the strong convergence rate of the proposed stochastic SAV {\color{black} scheme}s under Assumption \ref{ap1}. 

\begin{tm}
	\label{tm;1}
	Let $X_0\in \mathbb H^1$ and Assumption \ref{ap1}  hold. Suppose that 
	\begin{align}\label{con-tm1}
		&\|(-\Lambda)^{-\frac 12}\left(f(u)-f(\widetilde u)\right)	\|_{L^2}+\left\|(-\Lambda)^{-\frac 12}\left(g(\Theta u)-g(\Theta \widetilde u)\right)	\right\|_{\LL_2(\mathbf Q^\frac 12(\dot{\H}),\dot{\H})}\le C\|u-\widetilde u\|_{L^2},
	\end{align}
	where $u,\widetilde u\in \dot{\H}^1.$
	Then for $p\geq1$, the {numerical scheme} \eqref{nummet;SAV2} satisfies 
	\begin{align}
		\sup\limits_{n\in \{1,\ldots,N\}}\E\left[\|X(t_n)-X_n\|_{\mathbb H}^{2p}\right]\leq C\tau^{2p},
	\end{align}
	where $C:=C(p,X_0,\Q,T)>0$, $N\in \N^+$,  $N\tau=T.$
\end{tm}

\begin{proof}
	We only present the proof for the case $p=1$ since the proof for other cases is similar. 	Fix $n\in\{1,\ldots,N\}.$  Notice that the solution of \eqref{nummet;SAV2} can be rewritten as
	\begin{equation}
	\begin{split}
	X_n=&E(t_n)X_0
	+\sum\limits_{j=0}^{n-1}
	E(t_{n-j})
	\begin{bmatrix}
	0\\
	g(\Theta u_j)\delta W_j
	\end{bmatrix}\\
	&+\sum\limits_{j=0}^{n-1}
	E(t_{n-1-j})A^{-1}(E(\tau)-I)
	\begin{bmatrix}
	0\\
	-\frac{f(\widehat u_j)}{\sqrt{F(\widehat u_j)+\delta_0}}\frac{q_j+q_{j+1}}2
	\end{bmatrix}.
	\end{split}
	\end{equation}
	Recall that the mild solution of \eqref{mod;swe} satisfies 	\begin{equation}
	\begin{split}
	X(t_{n})=E(t_n)X_0&+
	\sum\limits_{j=0}^{n-1}\int_{t_j}^{t_{j+1}}
	E(t_{n}-s)[
	\F(X(s))ds+\G(X(s))dW(s)].
	\end{split}
	\end{equation}
	Let $\varepsilon_i=X(t_{i})-X_{i}$ for $i\in\{0,1,\ldots,N\}$. 
	\sly{Then} 
	\begin{align*}
		\|\varepsilon_{n}\|_{\mathbb H}^{2}\leq
		&C\Big\|\sum\limits_{j=0}^{n-1}
		\int_{t_j}^{t_{j+1}}
		(E(t_{n}-s)-E(t_{n}-t_j))
		\begin{bmatrix}
			0\\
			g(\Theta u(s))dW(s)
		\end{bmatrix}
		\Big\|_{\mathbb H}^2
		\\
		&+C\Big\|\sum\limits_{j=0}^{n-1}
		\int_{t_j}^{t_{j+1}}
		E(t_{n}-t_j)\begin{bmatrix}
			0\\
			(g(\Theta u(s))-g(
			\Theta u_j))dW(s)
		\end{bmatrix}
		\Big\|_{\mathbb H}^{2}
		\\
		&+
		C\Big\|\sum\limits_{j=0}^{n-1}
		\int_{t_j}^{t_{j+1}}
		\Big(E(t_n-s)
		\begin{bmatrix}
			0\\
			-\frac{f(u(s))q(s)}{\sqrt{F(u(s))+\delta_0}}
			+\frac{f(\widehat u_j)}{\sqrt{F(\widehat u_j)+\delta_0}}\frac{q_j+q_{j+1}}2
		\end{bmatrix}
		\Big)
		ds\Big\|_{\mathbb H}^{2}\\
		=:&Err_{n,1}+Err_{n,2}+Err_{n,3}.
	\end{align*}
	Then by Assumption \ref{ap1},  Lemmas \ref{lm;eu} and \ref{lm;prpE_n}, we obtain
	\begin{align*}
		\mathbb E[Err_{n,1}]
		\leq &
		C \int_{0}^{t_n}
		\mathbb E\left\|\left(E(s-\left[\frac{s}{\tau}\right]\tau)-I\right)
		\begin{bmatrix}
			0\\
			g(\Theta u(s))
		\end{bmatrix}
		\right\|_{\LL_2(\mathbf Q^\frac 12(\H), \H)}^{2}
		ds\\
		\leq& C\sum\limits_{j=0}^{N-1} 
		\int_{t_j}^{t_{j+1}} 
		\big|s-\left[\frac{s}{\tau}\right]\tau\big|^{2}
		\mathbb E\left[\left\|g(\Theta u(s))
		\right\|_{\LL_2(\mathbf Q^\frac 12(\dot{\H}), \dot{\H})}^{2}\right]ds\leq C\tau^2. 
	\end{align*}
	Similar arguments, together with Burkholder--Davis--Gundy's inequality, yield that 
	\begin{align*}
		\mathbb E[Err_{n,2}]
		\leq & C
		\mathbb E \left\| 
		\int_{0}^{t_n}
		E\left(t_n-\left[\frac{s}{\tau}\right]\tau\right)\begin{bmatrix}
			0\\
			(g(\Theta u(s))-g(\Theta u([\frac{s}{\tau}]\tau))dW(s)
		\end{bmatrix}
		\right\|_{\mathbb H}^2\\
		&+C\mathbb E \left\| 
		\int_{0}^{t_n}
		E\left(t_n-\left[\frac{s}{\tau}\right]\tau\right)\begin{bmatrix}
			0\\
			(g(\Theta u([\frac{s}{\tau}]\tau))-g(\Theta u_{[\frac{s}{\tau}]})dW(s)
		\end{bmatrix}
		\right\|_{\mathbb H}^2\\
		\leq &
		C\int_{0}^{t_n}
		\mathbb E\left\|E(t_n-\left[\frac{s}{\tau}\right]\tau)
		\begin{bmatrix}
			0\\
			g(\Theta u(s))-g(\Theta u(\left[\frac{s}{\tau}\right]\tau)
		\end{bmatrix}
		\right\|_{\LL_2(\mathbf Q^\frac 12(\H), \H)}^2
		ds\\
		&+
		C\int_{0}^{t_n}
		\mathbb E\left\|E(t_n-\left[\frac{s}{\tau}\right]\tau)
		\begin{bmatrix}
			0\\
			g(\Theta u([\frac{s}{\tau}]\tau))-g(\Theta u_{[\frac{s}{\tau}]})
		\end{bmatrix}
		\right\|_{\LL_2(\mathbf Q^\frac 12(\H), \H)}^{2}
		ds.
	\end{align*}
	Thanks to \eqref{con-tm1} and Proposition \ref{holder}, we have
	\begin{align*}
		\mathbb E[Err_{n,2}]
		\leq &
		C
		\int_{0}^{t_n}
		\mathbb E\left\|(-\Lambda)^{-\frac 12}\left(g(u(s))-g\left(u\left(\left[\frac{s}{\tau}\right]\tau\right)\right)\right)
		\right\|_{\LL_2(\mathbf Q^\frac 12(\dot{\H}),\dot{\H})}^2ds\\
		&+C
		\int_{0}^{t_n}
		\mathbb E\left\|(-\Lambda)^{-\frac 12}\left(g\left(u\left(\left[\frac{s}{\tau}\right]\tau\right)\right)-g(u_{\left[\frac{s}{\tau}\right]})\right)
		\right\|_{\LL_2(\mathbf Q^\frac 12(\dot{\H}),\dot{\H})}^2ds\\
		\leq & C\sum\limits_{j=0}^{n-1}\int_{t_j}^{t_{j+1}}
		\mathbb E\|u(s)-u(t_j)\|_{L^2}^2ds
		+C\sum\limits_{j=0}^{n-1}\int_{t_j}^{t_{j+1}}
		\mathbb E\|u(t_j)-u_j\|_{L^2}^2ds
		\\
		\leq& C\tau^2+ 
		C 
		\tau \sum\limits_{j=0}^{n-1} \mathbb E
		\left\|
		u(t_j)-u_j
		\right\|_{L^2}^2.
	\end{align*}
	Now we turn to consider $Err_{n,3}.$ Based on Lemma \ref{lm;2}, we obtain
	{\small
		\begin{align}\label{EST;Err_n,3}
			&\mathbb E[Err_{n,3}]\\\nonumber
			\leq 
			&C\mathbb E\left[\sum\limits_{j=0}^{n-1}\int_{t_j}^{t_{j+1}}\left\| 
			f(u(s))-\frac{f(\widehat u_j)q_j}{\sqrt{F(\widehat u_j)+\delta_0}}
			\right\|_{\dot{\mathbb H}^{-1}}^2 ds\right]+C\mathbb E\left[\sum\limits_{j=0}^{n-1}\int_{t_j}^{t_{j+1}}\left\|\frac{f(\widehat u_j)(q_{j+1}-q_j)}{2\sqrt{F(\widehat u_j)+\delta_0}}
			\right\|_{\dot{\mathbb H}^{-1}}^2ds\right]\\ \nonumber 
			\leq & C\mathbb E\left[\sum\limits_{j=0}^{n-1}\int_{t_j}^{t_{j+1}}\left\| 
			f(u(s))-f(u(t_j))
			\right\|_{\dot{\mathbb H}^{-1}}^2 ds\right]+C\mathbb E\left[\sum\limits_{j=0}^{n-1}\int_{t_j}^{t_{j+1}}\left\| 
			f(u(t_j))-f(\widehat u_j)
			\right\|_{\dot{\mathbb H}^{-1}}^2 ds\right]\\\nonumber
			&+C\mathbb E\left[\sum\limits_{j=0}^{n-1}\int_{t_j}^{t_{j+1}}\left\| 
			f(\widehat u_j)-\frac{f(\widehat u_j)q_j}{\sqrt{F(\widehat u_j)+\delta_0}}
			\right\|_{\dot{\mathbb H}^{-1}}^2 ds\right]+C\mathbb E\left[\sum\limits_{j=0}^{n-1}\int_{t_j}^{t_{j+1}}\left\| 
			\frac{f(\widehat u_j)(q_{j+1}-q_j)}{\sqrt{F(\widehat u_j)+\delta_0}}
			\right\|_{\dot{\mathbb H}^{-1}}^2 ds\right]\\\nonumber
			=:&\MyRoman{2}_{n}^{1}+\MyRoman{2}_{n}^{2}+\MyRoman{2}_{n}^{3}+\MyRoman{2}_{n}^{4}.
		\end{align}
		\normalsize{A}}ccording  to  \eqref{con-tm1} and Proposition \ref{holder}, we derive
	\begin{align*}
		\MyRoman{2}_{n}^{1}
		\leq& C \mathbb E\left[\sum\limits_{j=0}^{n-1}\int_{t_j}^{t_{j+1}}\|u(s)-u(t_j)\|_{L^2}^2ds\right]\leq C\tau^2,\\
		\MyRoman{2}_{n}^{2}
		\leq & C\tau\sum\limits_{j=0}^{n-1}\mathbb E\left[\|u(t_j)-u_j\|_{L^2}^2\right]
		+C\tau \sum\limits_{j=0}^{n-1}\mathbb E\left[\|u_j-\widehat u_j\|_{L^2}^2\right].
	\end{align*}
	If $\widehat u_j=u_j,$ then $\MyRoman{2}_{n}^{2}\leq C \tau \sum\limits_{j=0}^{n-1}\mathbb E\left[\|u(t_j)-u_j\|_{L^2}^2\right].$ When $\widehat u_j=\frac{3u_j-u_{j-1}}2,$ thanks to $u_0=u_{-1},$ we have 
	\begin{align*}
		\MyRoman{2}_{n}^{2}
		\leq & C\tau \sum\limits_{j=0}^{n-1}\mathbb E\left[\|u(t_j)-u_j\|_{L^2}^2\right]
		+C\tau \sum\limits_{j=0}^{n-1}\mathbb E\left[\|u_{j}-u_{j-1}\|_{L^2}^2\right]\\
		\leq & C\tau \sum\limits_{j=0}^{n-1}\mathbb E\left[\|u(t_j)-u_j\|_{L^2}^2\right]
		+C\tau \sum\limits_{j=0}^{n-2}\mathbb E\left[\|u_{j+1}-u_{j}\|_{L^2}^2\right].
	\end{align*}
	For the term $\MyRoman{2}_{n}^{4},$  we obtain
	\begin{align*}
		\MyRoman{2}_{n}^{4}\leq & C \mathbb E\left[\sum\limits_{j=0}^{n-1}\int_{t_j}^{t_{j+1}}
		\frac{\|f(\widehat u_j)\|_{\dot{\mathbb H}^{-1}}^2}{F(\widehat u_j)+\delta_0}
		|q_{j+1}-q_j|^2ds\right]\\
		\leq & C\mathbb E\left[\sum\limits_{j=0}^{n-1}\int_{t_j}^{t_{j+1}}
		\frac{\|f(\widehat u_j)\|_{\dot{\mathbb H}^{-1}}^2}{(F(\widehat u_j)+\delta_0)^2}
		\|f(\widehat u_j)\|_{L^2}^2\|u_{j+1}-u_j\|_{L^2}^2
		ds\right].
	\end{align*}
	Since $\|f(\widehat u_j)\| \leq b_1(\|\widehat u_j\|_{\dot{\H}^1})$ and $F(\widehat u_j)+\delta_0\geq c>0,$ the H\"older's inequality, Lemma \ref{nu;holder} and Corollary \ref{pri-num} yield that 
	\begin{align*}
		\MyRoman{2}_{n}^{4}\leq C\tau \sum\limits_{j=0}^{n-1}\sqrt{\mathbb E\left[\|u_{j+1}-u_j\|_{L^2}^4\right]}\le C\tau^2.
	\end{align*}	
	By applying Proposition \ref{aux-prop}, Corollary \ref{pri-num} and H\"older's inequality, we get 
	\begin{align*}
		\MyRoman{2}_{n}^{3}\leq & C\mathbb E\left[\sum\limits_{j=0}^{n-1}\int_{t_j}^{t_{j+1}}\left\| 
		f(\widehat u_j)-\frac{f(\widehat u_j)q_j}{\sqrt{F(\widehat u_j)+\delta_0}}
		\right\|_{\dot{\mathbb H}^{-1}}^2 ds\right]\\
		\leq & C\mathbb E\left[\sum\limits_{j=0}^{n-1}\int_{t_j}^{t_{j+1}}
		\frac{\|f(\widehat u_j)\|_{\dot{\mathbb H}^{-1}}^2}{F(\widehat u_j)+\delta_0}
		|\sqrt{F(\widehat u_j)+\delta_0}-q_j|^2 ds\right]\\
		\leq &C\tau^2. 
	\end{align*}
	Combining the above estimates together, we arrive at
	\begin{align*}
		\mathbb \|\varepsilon_{n}\|_{\mathbb H}
		^2\leq& C\tau^2+ 
		C 
		\tau \sum\limits_{j=0}^{n-1} \mathbb E
		\left\|
		u(t_j)-u_j
		\right\|_{L^2}^2.
	\end{align*}
	By discrete Gr\"onwall's inequality and Lemma \ref{nu;holder}, we complete the proof.
\end{proof}

Next, we present the strong convergence rate of \eqref{nummet;SAV1}, which can be shown by using the similar steps as in Theorem \ref{tm;1},  \cite[Lemmas 4.1 and 4.2]{CHJS19} and Proposition \ref{aux-prop}.  We omit its proof for convenience.

\begin{tm}\label{tm-2}
	Let $X_0\in \mathbb H^\gamma, \gamma\in [1,2],$ and Assumption \ref{ap1}  hold. Suppose that 
	\begin{align}\nonumber
		&\|(-\Lambda)^{-\frac 12}\left(f(u)-f(\widetilde u)\right)	\|_{L^2}+\left\|(-\Lambda)^{-\frac 12}\left(g(\Theta u)-g(\Theta \widetilde u)\right)	\right\|_{\LL_2(\mathbf Q^\frac 12(\dot{\H}),\dot{\H})}\le C\|u-\widetilde u\|_{L^2},\\\label{con-tm2}
		& \|(-\Lambda)^{\frac {\gamma-1}2} f(u)\|_{L^2}+\|(-\Lambda)^{\frac {\gamma-1}2} g(u)\|_{\mathcal L_2(\mathbf Q^{\frac 12}(\dot{\H}),\dot{\H})} \le b_6 (\|u\|_{\dot{\H}^1}),
	\end{align}
	where $b_6$ is a polynomial, $u,\widetilde u\in \dot{\H}^1.$
	Then for $p\geq1$, the {numerical scheme} \eqref{nummet;SAV1} satisfies
	\begin{align}
		\sup\limits_{n\in \{1,\ldots,N\}}\E\left[\|X(t_n)-X_n\|_{\mathbb H}^{2p}\right]\leq C\tau^{p\gamma},
	\end{align}
	where $C:=C(p,X_0,\Q,T)>0$, $N\in \N^+$,  $N\tau=T.$
\end{tm}

\begin{rk}
	\sly{ Schemes \eqref{nummet;SAV1} and \eqref{nummet;SAV2} can also be used to numerically solve the deterministic wave equation, which corresponds to the case of $g=0$. 
		In this case, the above convergence result of SAV {\color{black} scheme}s is also new. }
\end{rk}

\section{Semi-implicit energy-preserving {fully-discrete schemes}}
\label{sec5}
In this section, we combine the stochastic SAV {\color{black} scheme}s with a spatial finite element method (see, e.g., \cite{ACLW16,KLS10}) to propose implementary and energy-preserving fully-discrete schemes for \eqref{mod;swe1}. We first briefly recall the definition of the linear finite element method.
Let $\left\{\mathcal{T}_{h}\right\}$ be a quasi-uniform family of triangulations of the convex polygonal domain $\mathcal O$ with $h_{\mathrm{K}}=\operatorname{diam}(K)$ and $h=\max\limits_{\mathrm{K} \in \mathcal{T}_{h}} h_{\mathrm{K}} .$ 
Let $V_{h} \subset\dot{\mathbb H}^{1}$ be the
space of piecewise linear continuous functions with respect to $\mathcal{T}_{h}$ which are zero on the boundary of $\mathcal{D}$, and let $\mathcal{P}_{h}: {\dot{\H}} \rightarrow V_{h}$ denote the $\dot{\H}$-orthogonal projector and $\mathcal{R}_{h}: \dot{\mathbb H}^{1} \rightarrow V_{h}$ denote the $\dot{\mathbb H}^{1}$-orthogonal projector (Ritz projector). Thus,
\begin{align*}
	\left(\mathcal{P}_{h} v, w_{h}\right)=\left(v, w_{h}\right), \quad\left(\nabla \mathcal{R}_{h} u, \nabla w_{h}\right)=\left(\nabla u, \nabla w_{h}\right) \quad \forall \; v \in \dot{\H}, u \in \dot{\mathbb H}^{1}, w_{h} \in V_{h}.
\end{align*}
The discrete Laplace operator $\Lambda_{h}: V_{h} \rightarrow V_{h}$ is then defined by
\begin{align*}
	\left(\Lambda_{h} v_{h}, w_{h}\right)=-\left(\nabla v_{h}, \nabla w_{h}\right) \quad \forall\; w_{h} \in V_{h}.
\end{align*}
Notice that $\mathcal{R}_{h}=(-\Lambda_{h})^{-1} \mathcal{P}_{h} (-\Lambda)$ (see, e.g., \cite{KLS10}). We  define discrete norms and interpolation spaces by 
\begin{align*}
	\left\|v_{h}\right\|_{h, \alpha}=\left\|(-\Lambda_{h})^{\frac \alpha 2} v_{h}\right\|\quad \forall\; v_{h} \in V_{h}
\end{align*}
and $\dot{\mathbb H}_{h}^{\alpha}=V_{h}$ equipped with the norm $\|\cdot\|_{h, \alpha}$, {respectively}. 
Then the finite element method of \eqref{mod;swe1} becomes
\begin{equation}
\label{fem}
\left\{
\begin{aligned}
&d X^{h}(t)=A_{h} X^{h}(t) d t+\mathcal{P}_{h} \mathbb F\left(X^{h}(t)\right) d t+\mathcal{P}_{h} \mathbb G\left(X^{h}(t)\right) d W(t), \quad t>0,\\
&X^{h}(0)=X^{h, 0},
\end{aligned}
\right.
\end{equation} 
where $X^{h,0}=(u^{h,0},v^{h,0})^{\top},$ $u^{h,0}=\mathcal R_h u_0, v^{h,0}=\mathcal P_h v_0$,
\begin{align*}
	A_{h}:=
	\begin{bmatrix}
		0 & I \\
		\Lambda_{h} & 0
	\end{bmatrix}  \quad\text{and} \quad
	X^{h}:=\begin{bmatrix}
		u_{h} \\
		v_{h}
	\end{bmatrix}.
\end{align*}
Here the notation $\mathcal P_h \mathbb F(X^h)=(0,\mathcal P_h f(u^{h}))^{\top}$ and similarly for $\mathcal P_h \mathbb G(X^h).$ Like \eqref{sol;swe}, the mild form of  $X^h$  reads
\begin{align*}
	X^h(t)=E_h(t)X^{h,0}+\int_{0}^tE_h(t-s)\mathcal P^h\mathbb F(X^h(s))ds+\int_0^tE_h(t-s) \mathcal P^h \mathbb G(X^h(s))dW(s),
\end{align*}
where $E_h=\begin{bmatrix}
C_h(t) & (-\Lambda_h)^{-\frac 12} S_h(t)\\
-(-\Lambda_{h})^{\frac 12}S_h(t) & C_h(t)
\end{bmatrix}$ is a $\mathbb C_0$-semigroup generated by $A_h$ on $\mathbb H_h:=\dot{\H}^0_h\times \dot{\H}^{-1}_h$ with  $C_h(t)=\cos(t(-\Lambda_h)^{\frac 12})$ and  $S_h(t)=\sin(t (-\Lambda_h)^{\frac 12}).$ For \eqref{fem}, the corresponding energy is defined by
\begin{align*}
	H(X^{h})=\frac{1}{2}\left\|(-\Lambda_{h})^{\frac 12} u^{h}\right\|^{2}+\frac{1}{2}\left\|v^{h}\right\|^{2}+F(u^h), 
\end{align*}
since $\left\|\nabla u^{h}\right\|^2=\left\|(-\Lambda_h)^{\frac 12} u^{h}\right\|^2$.  By using the It\^o formula, 
one can obtain the following averaged energy evolution of the finite element solution $X^{h}$. 
Similar to the continuous case, the regularity estimate of the mild solution follows.

\begin{prop} \label{FEM-ene}
	Let  Assumption \ref{ap2}  hold and $X_0\in \H^1$.  
	The solution $X^{h}$ of the finite element approximation \eqref{fem} satisfies the averaged energy evolution law,
	\begin{align*}
		\mathbb{E}\left[H(X^{h}(t))\right]=&\mathbb{E}\left[H(X^{h}(0))\right]\\
		&+\frac{1}{2}\int_0^t \E\Big[ \operatorname{Tr}\left(\mathcal{P}_{h}g(\Theta u^h(s){\bf Q}^{\frac 12}) (\mathcal{P}_{h}g(\Theta u^h(s)){\bf Q}^{\frac 12})^*\right)\Big ]ds, \quad t \geq 0.
	\end{align*}
	Furthermore, suppose that $X_0\in \H^\beta$ for some $\beta\in [1,2]$ and that
	\begin{align}\label{con-fg1}
		&\|(-\Lambda)^{\frac {\beta-1}2} f(u)\|_{L^2}+\|(-\Lambda)^{\frac {\beta-1}2} g(\Theta u)\|_{\mathcal L_2(\mathbf Q^{\frac 12} (\dot{\mathbb H}), \dot{\mathbb H})}\le b_4(\|u\|_{\dot{\H}^1}), \; u \in \dot{\H}^\beta,
	\end{align}
	where  $b_4$ is a polynomial. Then it holds that for any $p\ge 2,$
	\begin{align*}
		\sup_{t\in [0,T]}\E \Big[\|X^h(t)\|_{\H_h^{\beta}}^p\Big]\le C(p,X_0,\Q,T),
	\end{align*}
	and for $0\le s\le t\le T,$
	\begin{align*}
		&\E \Big[\|u^h(t)-u^h(s)\|_{\dot{\H}_h^{\small 0}}^p\Big]\le C(p,X_0,\Q,T)|t-s|^p, \\
		&\E \Big[\|v^h(t)-v^h(s)\|_{\dot{\H}_h^{\small-1}}^p\Big]\le C(p,X_0,\Q,T)|t-s|^{\frac {p}2},
	\end{align*}
	{ where  $C(p,X_0,\Q,T)$ is a positive constant depending on $p,X_0,\Q,$ and $T$.} 
\end{prop}

Next, we apply the SAV {\color{black} scheme} \eqref{nummet;SAV2}
to  \eqref{fem} and obtain the following fully-discrete scheme 
\begin{equation}
\label{nummet;SAV2-full}
\begin{aligned}
&X^h_{n+1}=E_h(\tau)X^h_{n}+A_h^{-1}(E_h(\tau)-I)\mathcal P_h \Big(0,-\frac{f(\widehat u^h_n)}{\sqrt{F(\widehat u^h_n)+\delta_0}}\frac{q_n+q_{n+1}}2\Big)^\top\\
&\quad\quad\quad+
E_h(\tau)\mathcal P_h(0,g(\Theta u^h_n)\delta W_n)^\top,\\
&q_{n+1}=q_n+\frac{(f(\widehat u_n^h), u_{n+1}^h-u_n^h)}{2 \sqrt{F(\widehat u_n^h)+\delta_0}}.
\end{aligned}
\end{equation}
Here $\widehat u_n^h=u_n^h$ \sly{(or $\frac {3u_n^h-u_{n-1}^h}2$, with $u_{-1}^h=u_0^h$)}, $T>0$, $N\in \N^+$,  $N\tau=T, n\le N$.
Following the procedures in the proof of Theorem \ref{tm;1}, we could get the following strong convergence rate result of \eqref{nummet;SAV2-full}. Below we only present a sketch proof due to the limitation of pages.

\begin{tm}\label{err-fem-sav}
	Under Assumption \ref{ap1}, let $T>0$, $N\in \N^+$,  $N\tau=T, n\le N$. Assume that $f,g$ satisfy \eqref{con-fg1}  and \eqref{con-tm1} with some $\beta\ge 1$,
	and that $X_0\in \mathbb H^{\gamma},$ $\gamma\ge 1+\frac 23 \beta.$
	Then the {numerical scheme} \eqref{nummet;SAV2-full} satisfies that for $p\ge 2,$
	\begin{align*}
		\E \Big[\|X^h_n-X(t_n)\|_{\mathbb H_h}^{p}\Big]\le C(p,X_0,\Q,T) (h^{\frac 23 \beta p}+\tau^p).
	\end{align*} 
\end{tm}
\begin{proof}
	
	We decompose $X^h_n-X(t_n)$ by 
	$$X^h_n-X(t_n)=X^h_n-X^h(t_n)+X^h(t_n)-X(t_n).$$
	The  estimate of $\E \Big[\|X^h(t_n)-X(t_n)\|_{\mathbb H_h}^{p}\Big]\le C(p,X_0,\Q,T) h^{\frac 23 \beta p}$ could be established by using similar arguments as in the proof of \cite[Theorem 3.1]{ACLW16}.
	For the {term} $X^h_n-X^h(t_n)$, one may follow same steps as in proving Theorem \ref{tm;1}, together with the properties of the discrete cosine and sine operators (see, e.g., \cite{ACLW16}), and obtain
	$$\E \Big[\|X^h_n-X^h(t_n)\|_{\mathbb H_h}^{p}\Big]\le C(p,X_0,\Q,T) \tau^{p}.$$
	Combining the above estimates, we complete the proof.
\end{proof}

Following {the above} approach, one could obtain the analogous result for the fully-discrete scheme based on  \eqref{nummet;SAV1} and linear finite element method, i.e.,
\begin{equation}
\label{nummet;SAV1-full}
\begin{aligned}
&u_{n+1}^h=u_n^h+\frac h2(v_n^h+v_{n+1}^h)+\frac h2 \mathcal P_h g(\Theta u^h_n)\delta W_n,\\
&v_{n+1}^h=v_n^h+\frac h2\Lambda_h(u_n^h+u_{n+1}^h)-h\frac{\mathcal P_hf(\widehat u_n^h)}{\sqrt{F(\widehat u_n^h)+\delta_0}}\frac{q_n+q_{n+1}}2+\mathcal P^h g(\Theta u_n^h)\delta W_n,\\
&q_{n+1}=q_n+\frac{(f(\widehat u_n^h), u^h_{n+1}-u^h_n)}{2 \sqrt{F(\widehat u_n^h)+\delta_0}},
\end{aligned}
\end{equation}
{where $\widehat u_n^h=u_n^h$ \sly{(or $\frac {3u_n^h-u_{n-1}^h}2$ with $u_{-1}^h=u_0^h$)}, $T>0$, $N\in \N^+$,  $N\tau=T, n\le N$.}
In the end of this section, we present the discrete energy evolution of \eqref{nummet;SAV2-full} and its weak error estimate. 

\begin{prop} \label{prop-weak}
	Let the condition of Theorem \ref{err-fem-sav} hold and $X_0\in \H^1$, $T>0$, $N\in \N^+$,  $N\tau=T, n\le N.$
	The proposed {\color{black} scheme} \eqref{nummet;SAV2-full} preserves the discrete averaged modified energy evolution law, i.e., for $n\in\{0,1,\ldots,N-1\},$
	\begin{align*}
		\mathbb E[\widetilde H(u_{n+1}^h,v_{n+1}^h,q_{n+1})]=&\mathbb E[\widetilde H(u_n^h,v_n^h,q_n)]+\frac \tau 2 \E \Big[{\rm{Tr}}
		\left(\mathcal P_h g(\Theta u_n^h)\Q^\frac 12(\mathcal P_hg( \Theta u_n^h)\Q^\frac 12)^*\right)\Big].
	\end{align*} 
	Here $\widetilde H(u_n^h,v_n^h,q_n)=\frac 12\|(-\Lambda_h)^{\frac 12}u_n^h\|^2+\frac 12\|v_n^h\|^2+q_n^{2}.$
	In addition, assume that 
	\begin{align}\label{weak-err}
		&\left\|\left(g(\Theta u)-g(\Theta \widetilde u)\right)\right\|_{\LL_2(\mathbf Q^\frac 12(\dot{\H}),\dot{\H})}\le C\|u-\widetilde u\|,\\\nonumber
		& \|(-\Lambda)^{\frac {\beta_1} 2}g(\Theta u)\|_{\LL_2(\mathbf Q^\frac 12(\dot{\H}),\dot{\H})}\le Cb_7(\|u\|_{\dot{\H}^{\beta_1}}),\quad  \beta_1\in [1,2].
	\end{align} 
	Then \eqref{nummet;SAV2-full} satisfies 
	\begin{align*}
		\Big|\E \Big[ \widetilde H(u_n^h,v_n^h,q_n)-V_1(u(t_n),v(t_n))\Big]\Big|\le C(p,T,X_0) (|err_0|+h^{\beta_1}+h^{\frac 23 \beta}+\tau),
	\end{align*} 
	where \sly{$\beta \in [1,3],$}  $err_0=\big|\widetilde H(u_0^h,v_0^h,q_0)-V_1(u(0),v(0))\big|.$ 
\end{prop}

\begin{proof}
	Following similar arguments as in the proof of Proposition \ref{ene-pre-1}, one can get the discrete energy evolution of \eqref{nummet;SAV2-full},
	\begin{align*}
		\mathbb E[\widetilde H(u_{n+1}^h,v_{n+1}^h,q_{n+1})]=&\mathbb E[\widetilde H(u_n^h,v_n^h,q_n)]+\frac \tau 2 \E \Big[{\rm{Tr}}
		\left(\mathcal P^h g(\Theta u^h_n)\Q^\frac 12 (\mathcal P_h g( \Theta u_n)\Q^\frac 12)^*\right)\Big],
	\end{align*}
	which {implies} the first assertion. 
	According to \eqref{evo-law-ave}, we have 
	\begin{align*}
		&\E \Big[ \widetilde H(u_n^h,v_n^h,q_n)-V_1(u(t_n),v(t_n))\Big]\\=&err_0+\frac 12\int_0^{t_n}\E \Big[{\rm{Tr}}
		\left(\mathcal P_h g(\Theta {u_{[\frac s \tau]}^h})\Q^\frac 12(\mathcal P_hg( \Theta {u_{[\frac s \tau]}^h})\Q^\frac 12)^*\right)\\
		&- {\rm{Tr}}
		\left( g(\Theta u(s))\Q^\frac 12(g( \Theta u(s))\Q^\frac 12)^*\right)\Big] ds.
	\end{align*}
	From \eqref{weak-err}, Proposition \ref{holder}, Corollary \ref{pri-num}, H\"older's inequality  and $\|(I-P_h)w\|\le C h^{\beta} \|w\|_{\dot{\H}^{\beta}}, \beta\in [1,2] $, it follows that
	\begin{align*}
		&\quad\Big|\E \Big[ \widetilde H(u_n^h,v_n^h,q_n)-V_1(u(t_n),v(t_n))\Big]\Big|\\
		&\le |err_0|
		+C{\int_{0}^{t_n}} \Big|\E \Big[{\rm{Tr}}
		\left((I-\mathcal P^h)g(\Theta u(s))\Q^\frac 12((I+\mathcal P_h)g( \Theta u(s))\Q^\frac 12)^*\right)\Big]\Big|ds\\
		&+C{\int_{0}^{t_n}} \Big|\E \Big[{\rm{Tr}}
		\left( \mathcal P^h \big(g(\Theta u(s))-g(\Theta {u_{[\frac s \tau]}^h})\big)\Q^\frac 12( \mathcal P_h \big(g(\Theta u(s))+g({u_{[\frac s \tau]}^h})\big)\Q^\frac 12)^*\right)\Big]\Big|ds\\
		&\le  |err_0|+ C{\int_{0}^{t_n}} \E \Big[ \Big\|(I-\mathcal P_h) g( \Theta u(s))\Big\|_{\LL_2(\mathbf Q^\frac 12(\dot{\H}),\dot{\H})} \Big\|g( \Theta u(s))\Big\|_{\LL_2(\mathbf Q^\frac 12(\dot{\H}),\dot{\H})} \Big]ds\\
		&+ C{\int_{0}^{t_n}} \E \Big[ \Big\| g( \Theta u(s))-g( \Theta {u_{[\frac s \tau]}^h})\Big\|_{\LL_2(\mathbf Q^\frac 12(\dot{\H}),\dot{\H})} \Big\|g( \Theta u(s))+g( \Theta {u_{[\frac s \tau]}^h})\Big\|_{\LL_2(\mathbf Q^\frac 12(\dot{\H}),\dot{\H})} \Big]ds\\
		&\le  |err_0|+Ch^{\beta_1}+\int_0^{t_n} \sqrt{\E\Big[ \|{u_{[\frac s \tau]}^h}-u(s)\|^2\Big]}ds.
	\end{align*}
	Applying the temporal regularity estimate in Proposition \ref{holder} and Theorem \ref{err-fem-sav}, we conclude that 
	\begin{align*}
		\Big|\E \Big[ \widetilde H(u_n^h,v_n^h,q_n)-V_1(u(t_n),v(t_n))\Big]\Big|
		&\le  |err_0|+Ch^{\beta_1}+C(\tau+h^{\frac 23\beta}).
	\end{align*}
\end{proof}

It can be seen that the weak error of the energy for
\eqref{nummet;SAV2-full} is only determined by the spatial discretization in the additive noise case since there is no weak error in the {temporal} direction in this case (see Proposition \ref{ene-pre-1}). 

\section{Numerical experiments}
\label{sec6}

This section presents numerical experiments to illustrate {the strong convergence order and \sly{the preservation of  energy evolution law} of the proposed numerical schemes \eqref{nummet;SAV2-full} and \eqref{nummet;SAV1-full} with $\hat u_n=u_n,$ $n\in\{0,1,\ldots, N\},$} for 1-dimensional nonlinear SWE under the homogeneous Dirichlet boundary condition,
\begin{equation} \label{eq6.2}
\left\{\begin{aligned}
&du = vdt,  &&(x, t) \in (0, 1) \times (0,T],\\
&dv = u_{xx} dt - f(u) dt +g(u) dW(t), &&(x, t) \in (0, 1) \times (0,T],\\
&u(x,0)=\sin(\pi x),\quad v(x,0)=0.
\end{aligned} \right. \end{equation}

\begin{figure}[h]
	\centering
	\subfigure{
		\begin{minipage}{5cm}
			\centering
			\includegraphics[height=4cm,width=5cm]{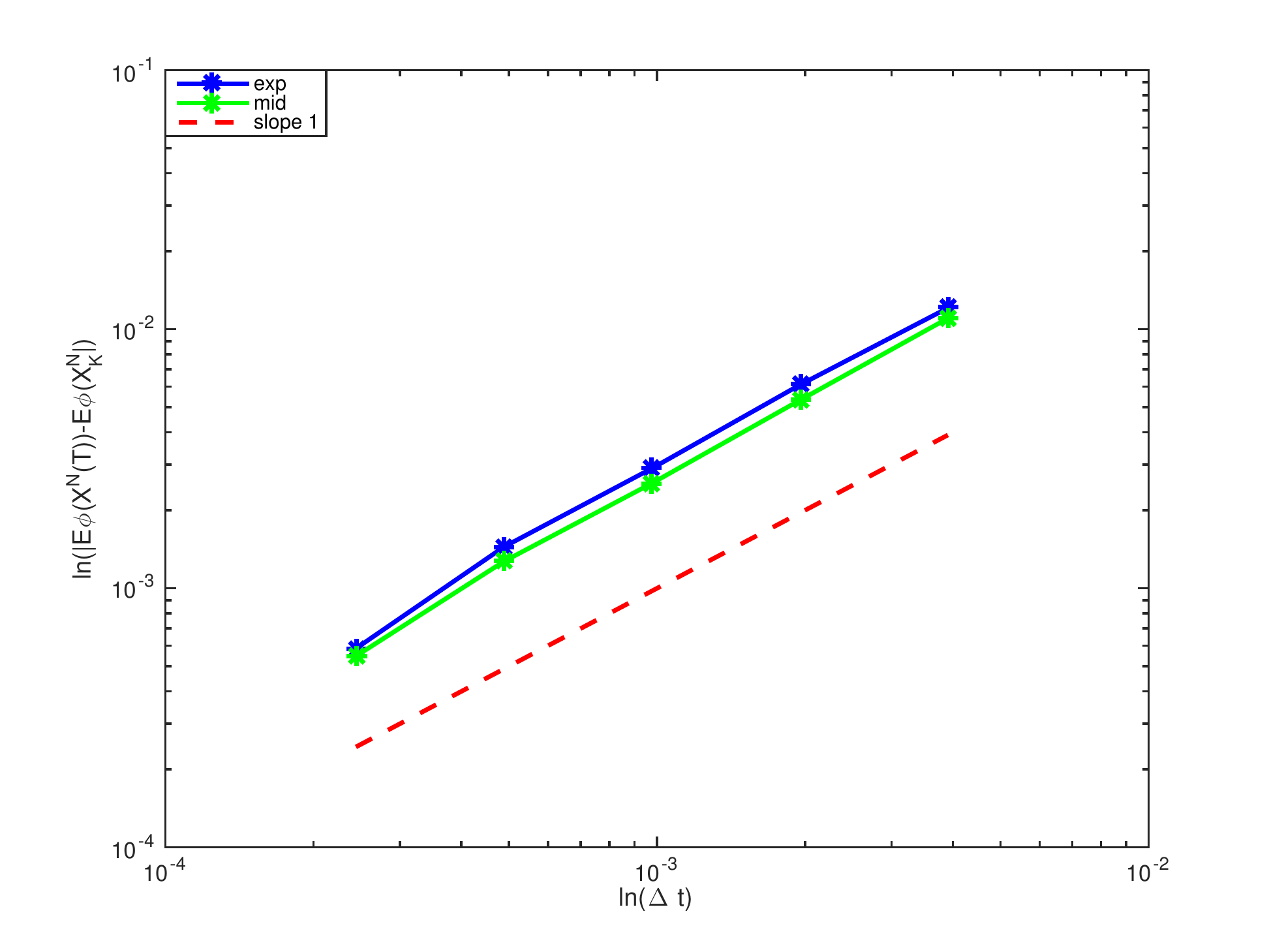}
		\end{minipage}
	}
	\subfigure{
		\begin{minipage}{5cm}
			\centering
			\includegraphics[height=4cm,width=5cm]{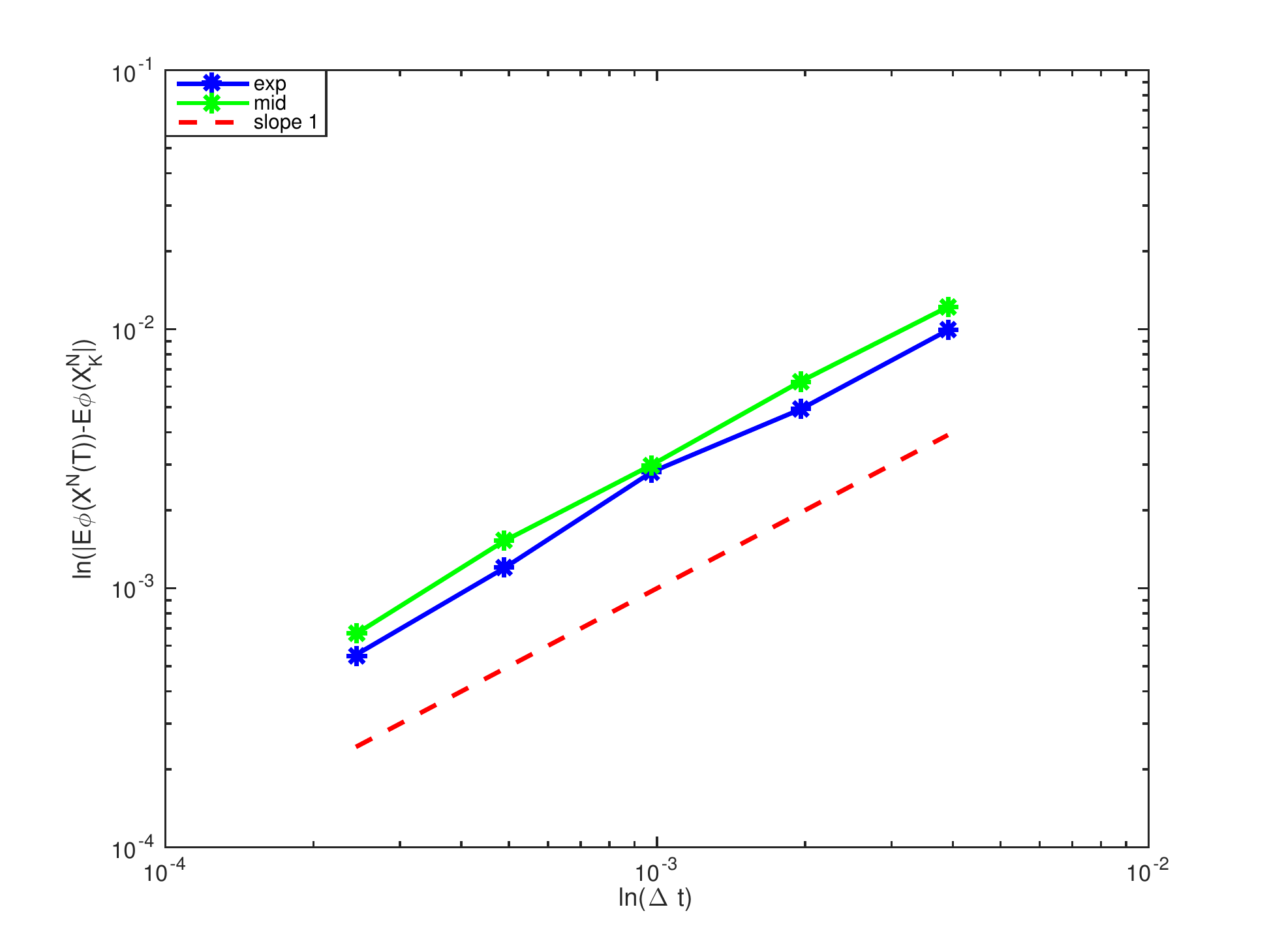}
		\end{minipage}
	}
	\caption{The strong convergence order in temporal direction $($left$)$ $f(u) = u, g(u) =\sin(u)$, $($right$)$ $f(u) = \sin(u), g(u) =\sin(u)$}
	\label{fig1}
\end{figure}

Fig. \ref{fig1} displays the temporal approximation errors $\|u(T)-U^h_N\|_{L^2}$ against $\tau$ on log-log scale with $\tau = 2^{-\bf s}, \,{\bf s} = 8,9,10,11,12$ at time $T=1$ for multiplicative noise, respectively. 
We fix the spatial step size $h=2^{-6}$ and simulate the exact solution with the numerical one by  using a small step size $\tau = 2^{-14}.$  
It can be observed that the slopes of numerical schemes are closed to 1, which implies that the temporal convergence order of the proposed numerical schemes  \eqref{nummet;SAV2-full} and \eqref{nummet;SAV1-full} is 1. Here the expectation is approximated by taking average over 1000 realizations.

\begin{figure}[h]
	\centering
	\subfigure{
		\begin{minipage}{5cm}
			\centering
			\includegraphics[height=4cm,width=5cm]{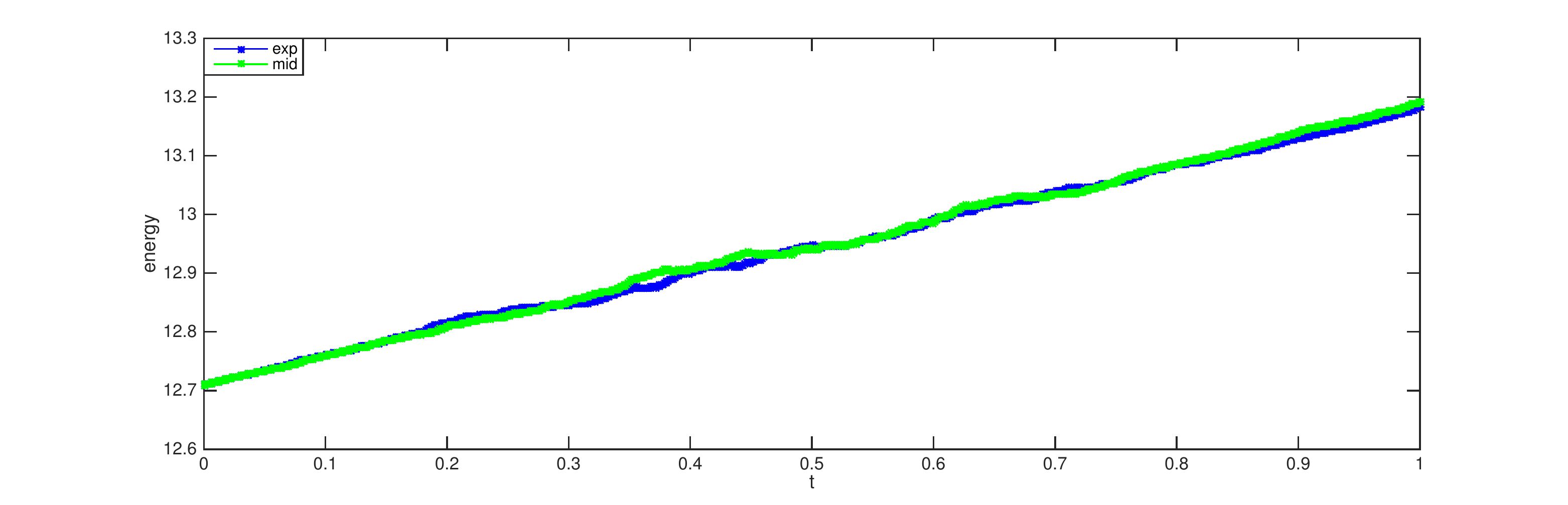}
		\end{minipage}
	}
	\subfigure{
		\begin{minipage}{5cm}
			\centering
			\includegraphics[height=4cm,width=5cm]{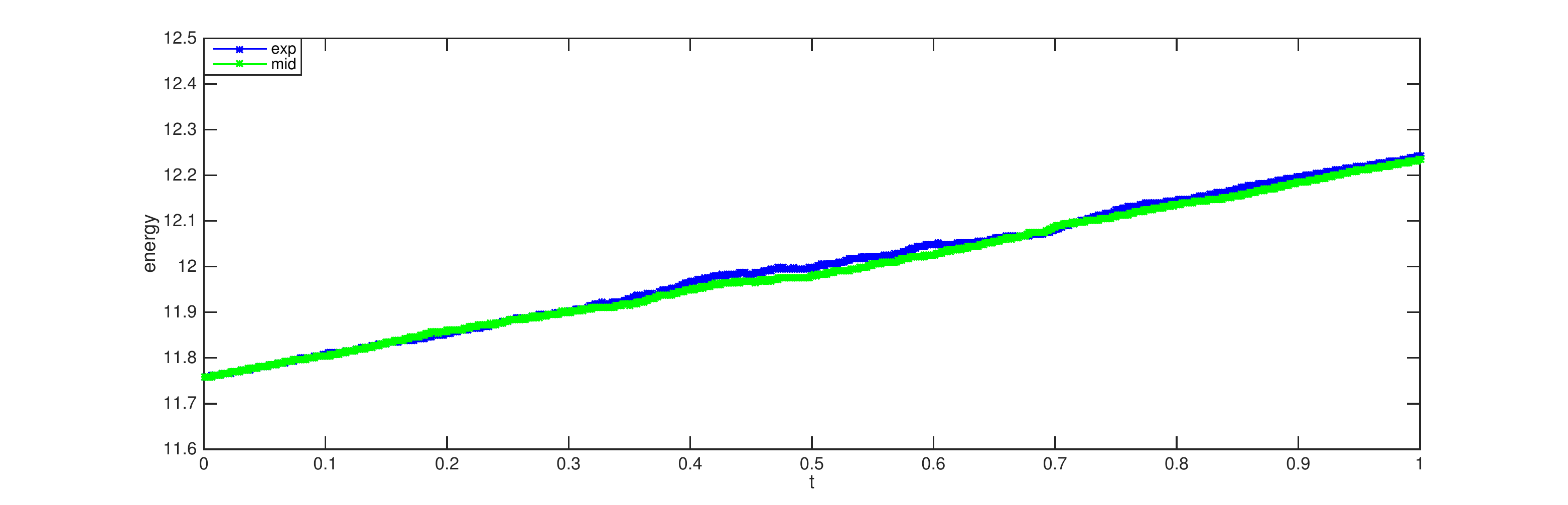}
		\end{minipage}
	}
	\caption{Averaged energy evolution relationship $($left$)$ $f(u) = u,$ $($right$)$ $f(u) = \sin(u)$}
	\label{fig2}
\end{figure}

Fig. \ref{fig2} presents the evolution of discrete averaged energies for the proposed  numerical schemes \eqref{nummet;SAV2-full} and \eqref{nummet;SAV1-full}. 
Here, the expectation is approximated by taking average over 5000 realizations.
From Fig. \ref{fig2}, it can be seen that averaged energies associated with numerical solutions  grow linearly with the time raising when the SWE is driven by additive noise.  
The numerical results are consistent with the theoretical results.

\bibliography{references}
\bibliographystyle{plain}
\end{document}